\documentclass[12pt,a4paper]{amsart}
\usepackage{bbm,graphicx}
\usepackage[square,comma]{natbib}
\title{Random Normal Matrices and Polynomial Curves}
\author{Peter Elbau}
\thanks{This work emerged from my PhD studies under the supervision of Prof. Giovanni Felder at the ETH Zurich}

\newtheorem{proposition}{Proposition}[section]
\newtheorem{corollary}[proposition]{Corollary}
\newtheorem{lemma}[proposition]{Lemma}
\newtheorem{theorem}[proposition]{Theorem}
\newtheorem{definition}[proposition]{Definition}
\newtheorem{conjecture}[proposition]{Conjecture}
\newcommand{\D}{\displaystyle}
\newcommand{\C}{\mathbbm C}
\newcommand{\R}{\mathbbm R}
\newcommand{\Z}{\mathbbm Z}
\newcommand{\N}{\mathbbm N}
\newcommand{\U}{\mathrm U}
\renewcommand{\u}{\mathfrak u}
\newcommand{\e}{\mathrm e}
\renewcommand{\i}{\mathrm i}

\renewcommand{\vec}[1]{\underline{#1}}
\newcommand{\ord}{\mathrm o}
\newcommand{\Ord}{\mathcal O}
\newcommand{\I}{^{-1}}
\newcommand{\0}{^{(0)}}
\renewcommand{\d}{\,\mathrm d}
\newcommand{\dd}{\,\mathrm d^2}
\newcommand{\ddd}[1]{\,\mathrm d^{2#1}}
\newcommand{\Int}[2]{\!\!\!\begin{array}[t]{ll}{\displaystyle\int #2} \\[-0.35em] {\scriptstyle\mbox{}\hspace*{0.3em} #1}\end{array}\!\!\!}

\DeclareMathOperator{\tr}{\mathrm{tr}}
\DeclareMathOperator{\diag}{\mathrm{diag}}
\DeclareMathOperator{\sign}{\mathrm{sign}}

\DeclareMathOperator{\RE}{\mathrm{Re}}
\DeclareMathOperator{\IM}{\mathrm{Im}}

\begin{document}
\begin{abstract}
We show that in the large matrix limit, the eigenvalues of the normal matrix model for matrices with spectrum inside a compact domain with a special class of potentials homogeneously fill the interior of a polynomial curve uniquely defined by the area of its interior domain and its exterior harmonic moments which are all given as parameters of the potential.

Then we consider the orthogonal polynomials corresponding to this matrix model and show that, under certain assumptions, the density of the zeros of the highest relevant orthogonal polynomial in the large matrix limit is (up to some constant factor) given by the discontinuity of the Schwarz function of this polynomial curve. 
\end{abstract}
\maketitle
\thispagestyle{empty}
\tableofcontents
\newpage

\section*{Introduction}
In 1999, P. Wiegmann and A. Zabrodin realised \cite{WZ,MWZ} that conformal maps from the exterior of the unit disc to the exterior of a simple closed analytic curve admit as functions of the exterior harmonic moments $(t_k)_{k=1}^\infty$ of the analytic curve the structure of an integrable hierarchy: the dispersionless Toda lattice hierarchy. (The Toda lattice hierarchy was introduced by K. Ueno and K. Takasaki in \cite{UT}. A review of the dispersionless Toda lattice hierarchy can be found in \cite{TT}.)

As was shown earlier by L.-L. Chau and Y. Yu \cite{CY,CZ}, the normal matrix model defined by the probability distribution 
\begin{equation}\label{e:NMMdefinition}
\mathcal P_{N,\mathcal V}(M)=\frac1{\mathcal Z_{N,\mathcal V}}\e^{-N\tr\mathcal V(M)},\quad \mathcal Z_{N,\mathcal V}=\int_{\mathcal N_N}\mathcal P_{N,\mathcal V}(M)\d M,
\end{equation}
on the space $\mathcal N_N$ of all normal $N\times N$ matrices with the potential 
\begin{equation}\label{e:NMMpotential}
\mathcal V(M)=\frac1{t_0}\left(MM^*-\sum_{k=1}^\infty(t_kM^k+\bar t_k{M^*}^k)\right)
\end{equation}
gives rise to a solution of the Toda lattice hierarchy. In \cite{KKMWZ}, this connection between normal matrix models and conformal maps was picked up and the authors showed 
that the eigenvalues in the large matrix limit homogeneously fill the interior of the analytic curve with the exterior harmonic moments $(t_k)_{k=1}^\infty$ and encircled area $\pi t_0$.

Nevertheless, all these results remain on the level of formal manipulations as in general already the integral in \eqref{e:NMMdefinition} diverges. 

The purpose of this work is now to proof these statements in a mathematically rigorous setting. To this end, we will first of all introduce a cut-off for the divergent integral by restricting ourselves to the space of those normal matrices whose spectrum lies in some compact domain. Furthermore, we will only consider polynomial potentials. Then, only finitely many exterior harmonic moments are different from zero and so the corresponding curve is polynomial. Since, at least for curves with small encircled area $\pi t_0$, we have a bijective relation between a polynomial curve and its harmonic moments, we finally can prove that, provided $t_0$ is small enough, in the large matrix limit the eigenvalues indeed homogeneously fill the interior of the polynomial curve uniquely determined by the exterior harmonic moments $(t_k)_{k=1}^\infty$ and the encircled area $\pi t_0$.

Moreover, we find that the distribution of the zeros of the $n$-th orthogonal polynomial $p_{n,N}$ associated to this matrix model is in the limit $\frac nN\to x\in[0,1]$, $N\to\infty$ (up to some constant) given by the discontinuity of the Schwarz function of the polynomial curve with area $\pi xt_0$ and exterior harmonic moments $(t_k)_{k=1}^\infty$, provided the zeros in this limit are confined to some one-dimensional tree-like graph.

The paper is organised as follows: 
In chaper one we will introduce the considered matrix model and review a few standard methods (see e.g. \cite{M}) to characterise the distribution of the eigenvalues. 
In the second chapter, we will show that polynomial curves with small area are uniquely determined by its exterior harmonic moments (for star-like domains this was already shown in \cite{VE}). 
Then, after having introduced the concept of the Schwarz function (see e.g. \cite{Da}), we will end up with the connection between harmonic moments of a polynomial curve and its parametrisation, which we will see is given by the dispersionless Toda lattice hierarchy. 
Thereafter, we will consider the large matrix limit, in chapter three from a general point of view which is well know from potential theory \cite{ST,J,D}, and in chapter four, we will apply those results to potentials of the form \eqref{e:NMMpotential} with only finitely many non-vanishing $t_k$, leading us to the desired result that for small $t_0$ the eigenvalues homogeneously fill the interior of the polynomial curve determined by the exterior harmonic moments $(t_k)_{k=1}^\infty$ and the area $\pi t_0$. 
In the last chapter, we will analyse the behaviour of the orthogonal polynomials of these potentials in the continuum limit and will find that they obey the Toda lattice hierarchy. 
Finally, we will connect the limiting distribution of the zeros of the $n$-th orthogonal polynomial to the discontinuity of the Schwarz function of the corresponding polynomial curve and calculate it explicitly for the Gaussian and the cubic case.

\section{The model}
\subsection{The eigenvalue distribution}
We start with the probability distribution
\[ \mathcal P_{N,\mathcal V}(M)=\frac1{\mathcal Z_{N,\mathcal V}}\e^{-N\tr\mathcal V(M)},\quad \mathcal Z_{N,\mathcal V}=\int_{\mathcal N_N(D)}\mathcal P_{N,\mathcal V}(M)\d M, \]
on the space
\[ \mathcal N_N(D)=\{M\in\C^{N\times N}\;|\;[M,M^*]=0,\;\sigma(M)\subset D\} \]
of all normal $N\times N$ matrices with spectrum in a compact subset $D\subset\C$, where $\mathcal V(M)$ denotes some polynomial in $M$ and $M^*$ such that $\mathcal V(M)=\mathcal V(M)^*$. The measure $\d M$ on the variety $\mathcal N_N(D)$ shall be the one induced by the flat metric on $\C^{N\times N}$.

\begin{proposition}
On the regular part of $\mathcal N_N(D)$, the measure $\d M$ factorises into 
\begin{equation}\label{e:normalMeasure}
\d M=\d U\prod_{1\le i<j\le N}|z_i-z_j|^2\prod_{i=1}^N\dd z_i,
\end{equation}
where $\{z_i\}_{i=1}^N$ is the spectrum of $M$, $M=U\diag(z_i)_{i=1}^NU^*$ for $U\in\U(N)/\U(1)^N$, and the measure
\begin{equation}\label{e:unitaryMeasure}
\d U=2^{\frac{N(N-1)}4}\prod_{i<j}\dd u_{ij},\quad \dd u_{ij}=(U^*)_{ik}\dd U_{kj},
\end{equation}
on $\U(N)/\U(1)^N$ is induced by the measure on $\U(N)$, which is given by the flat metric on $\C^{N\times N}$.
\end{proposition}
\begin{proof}
Pulling back the tangent spaces of $\U(N)$ to the Lie algebra $\u(N)$, we uniformly get the metric $g(X,Y)=\tr(X^*Y)$ on $\u(N)$. Since the tangent spaces of the orbits $U\cdot\U(1)^N$, $U\in\U(N)$, are generated by the set $\mathfrak d$ of diagonal matrices in $\u(N)$, the tangent spaces of $\U(N)/\U(1)^N$ are generated by $\u(N)/\mathfrak d$, and the metric $g$ reduces to
\[ \tilde g(X,Y)=2\sum_{i<j}\bar X_{ij}Y_{ij},\quad X,Y\in\u(N)/\mathfrak d. \]
With the upper triangular matrix elements as coordinates, we therefore get the measure \eqref{e:unitaryMeasure} on $\U(N)/\U(1)^N$.

Let now $t\mapsto M(t)$ be a smooth curve in the regular part of $\mathcal N_N(D)$ with $M(0)=M$. Then we find a smooth curve 
\begin{align*} 
&t\mapsto(\vec z(t),U(t))\in D^N\times \U(N)/\U(1)^N\quad\textrm{with}\\
&M(t)=U(t)D(t)U(t)^*,\quad D(t)=\diag(\vec z(t)),
\end{align*}
which up to a permutation of the eigenvalues is uniquely determined. Again taking the pull back of the tangent space of $\U(N)$, $\dot u=U^*\dot U$, the flat metric on $\C^{N\times N}$ reads
\[ \tr(\dot M^*\dot M)=\tr(\dot D^*\dot D+2\dot u(D^*\dot u-\dot uD^*)D)=\sum_{i=1}^N\dot{\bar z}_i\dot z_i+\sum_{i,j=1}^N\dot{\bar u}_{ij}\dot u_{ij}|z_i-z_j|^2. \]
Choosing $\dot z_i$ and $(\dot u_{ij})_{i<j}$ as coordinates, we get the Riemannian volume form \eqref{e:normalMeasure}.
\end{proof}

Since $\mathcal N_N(D)$ is regular almost everywhere, we can integrate over the group $\U(N)/\U(1)^N$, leaving us with the probability distribution 
\begin{equation}\label{e:eigenvalueDistribution}
P_{N,V}(\vec z) = \frac1{Z_{N,V}}\e^{-N\sum_{i=1}^NV(z_i)}|\Delta(\vec z)|^2,\quad Z_{N,V}=\left(\int_{D^N}P_{N,V}(\vec z)\ddd{N}{\vec z}\right)^{-1}, 
\end{equation}
for the eigenvalues $\vec z=(z_i)_{i=1}^N\in D^N$, where $\Delta(\vec z)=\det(z_i^{j-1})_{i,j=1}^N=\prod_{i<j}(z_j-z_i)$ denotes the van-der-Monde determinant and $V(z)$ is the real-valued polynomial in $z$ and $\bar z$ induced by $\mathcal V(z\,\mathbbm1)=V(z)\,\mathbbm1$, $z\in D$. From now on, as long as it is clear which potential $V$ we use, we will suppress it in the notation.

Because the probability $P_N$ vanishes if two eigenvalues are equal, we may consider it as probability distribution on the regular part 
\[ D^N_0 = \{\vec z\in D^N\;|\;z_i\ne z_j\;\forall\,i\ne j\}. \]

To describe the behaviour of the eigenvalues, we are going to introduce the notion of correlation functions.
\begin{definition}
The $k$-point correlation function $R^{(k)}_N$, $1\le k\le N$, is given by
\[ R^{(k)}_N\big((z_i)_{i=1}^k\big)=\frac{N!}{(N-k)!}\int_{D^{N-k}}P_N\big((z_i)_{i=1}^N\big)\prod_{i=k+1}^N\dd z_i. \]
So $\frac1{N!}R_N^{(N)}=P_N$, and $\frac1NR_N^{(1)}$ is the density of the eigenvalues.
\end{definition}

\begin{definition}
We denote by $A_N(n,\lambda,a)$, $0\le n\le N$, the probability of having exactly $n$ eigenvalues in a disc $B_\lambda(a)$ of radius $\lambda>0$ around some point $a\in D$:
\begin{equation}\label{e:levelSpacing} 
A_N(n,\lambda,a)=\begin{pmatrix}N\\n\end{pmatrix}\Int{B_\lambda(a)^n\times(D\setminus B_\lambda(a))^{N-n}}{P_N(\vec z)\ddd N{\vec z}.}
\end{equation}
\end{definition}

\begin{proposition}\label{p:levelSpacing}
For $0\le n\le N$, $\lambda>0$, and $a\in D$, we have
\[ A_N(n,\lambda,a)=\frac1{n!}\sum_{k=0}^{N-n}\frac{(-1)^k}{k!}\Int{B_\lambda(a)^{n+k}}{R_N^{(n+k)}(\vec z)\ddd{(n+k)}{\vec z}}. \]
\end{proposition}
\begin{proof}
Assume inductively that the equation
\begin{equation}\label{e:A_NInduction}
\Int{(D\setminus A)^n}{f(\vec z)\ddd n{\vec z}}=\sum_{k=0}^n(-1)^k\begin{pmatrix}n\\k\end{pmatrix}\Int{A^k\times D^{n-k}}{f(\vec z)\ddd n{\vec z}},
\end{equation}
which certainly is true for $n=0$, holds for some $A\subset D$ and any symmetric function $f$ on $D^n$ for some $n\in\N_0$. Then, for any symmetric function $f$ on $D^{n+1}$,
\begin{align*}
\Int{(D\setminus A)^{n+1}}{f(\vec z)\ddd{(n+1)}{\vec z}}&=\Int{D\times(D\setminus A)^n}{f(\vec z)\ddd{(n+1)}{\vec z}}-\Int{A\times(D\setminus A)^n}{f(\vec z)\ddd{(n+1)}{\vec z}} \\
&=\sum_{k=0}^n(-1)^k\begin{pmatrix}n\\k\end{pmatrix}\bigg(\Int{A^k\times D^{n-k+1}}{f(\vec z)\ddd{(n+1)}{\vec z}}-\Int{A^{k+1}\times D^{n-k}}{f(\vec z)\ddd{(n+1)}{\vec z}}\bigg) \\
&=\sum_{k=0}^{n+1}(-1)^k\begin{pmatrix}n+1\\k\end{pmatrix}\Int{A^k\times D^{n-k+1}}{f(\vec z)\ddd{(n+1)}{\vec z}},
\end{align*}
which proves the identity \eqref{e:A_NInduction} for all $n\in\N_0$ and any domain $A\subset D$.

Therefore,
\begin{align*}
A_N(n,\lambda,a)&=\begin{pmatrix}N\\n\end{pmatrix}\sum_{k=0}^{N-n}(-1)^k\begin{pmatrix}N-n\\k\end{pmatrix}\Int{B_\lambda(a)^{n+k}\times D^{N-n-k}}{P_N(\vec z)\ddd N{\vec z}} \\
&=\frac1{n!}\sum_{k=0}^{N-n}\frac{(-1)^k}{k!}\Int{B_\lambda(a)^{n+k}}{R_N^{(n+k)}(\vec z)\ddd{(n+k)}{\vec z}}.\qedhere
\end{align*}
\end{proof}

\subsection{Orthogonal polynomials}
One may formulate the problem of finding the correlation functions as the construction of a set of orthogonal polynomials.
\begin{definition}
We call the polynomials $p_{n,N}$ of degree $n$ with leading coefficient one, $n\in\N_0$, defined with the inner product 
\begin{equation}\label{e:innerProduct}
(f,g)_N = \int_D\overline{f(z)}g(z)\e^{-NV(z)}\dd z
\end{equation}
by the relation
\[ (p_{m,N},p_{n,N})_N=\delta_{m,n}h_{n,N}\quad\textrm{for all}\quad m,n\in\N_0 \]
for some positive numbers $h_{n,N}$, $n\in\N_0$, the orthogonal polynomials for the potential $V$ with norms $(h_{n,N})_{n=0}^\infty$. The normalised functions 
\[ q_{n,N}=\frac1{\sqrt{h_{n,N}}}\,p_{n,N},\quad n\in\N_0, \]
we call the orthonormal polynomials for the potential $V$.
\end{definition}
\begin{proposition}
Let $(p_{n,N})_{n=0}^\infty$ be the orthogonal polynomials for the potential $V$ with norms $(h_{n,N})_{n=0}^\infty$ and 
\[ K_N(w,z)=\e^{-\frac N2(V(w)+V(z))}\sum_{n=0}^{N-1}\frac{\overline{p_{n,N}(w)}p_{n,N}(z)}{h_{n,N}}. \]
Then
\begin{enumerate}
\item
$\D\int_DK_N(z,z)\dd z=N\quad\textrm{and}\quad\int_DK_N(w,v)K_N(v,z)\dd v=K_N(w,z)$,
\item
$\D Z_N=N!\prod_{n=0}^{N-1}h_{n,N}\quad\textrm{and}\quad R^{(k)}_N(\vec z)=\det\big(K_N(z_i,z_j)\big)_{i,j=1}^k$, $\vec z\in D^k$.
\end{enumerate}
\end{proposition}
\begin{proof}
The first part is a direct consequence of the orthonormality of the polynomials $(\frac1{\sqrt{h_{n,N}}}p_{n,N})_{n=0}^{N-1}$. For the second part, we write out the van-der-Monde determinants in equation \eqref{e:eigenvalueDistribution}, use the fact that the determinant is invariant under column operations, and get for $\vec z\in D^N$
\begin{align*}
P_N(\vec z)&=\frac{\prod_{n=0}^{N-1}h_{n,N}}{Z_N}\e^{-N\sum_{i=1}^NV(z_i)}\sum_{\sigma,\tau\in S_N}\sign(\tau\sigma)\prod_{i=1}^N\frac{\overline{p_{\sigma(i)-1,N}(z_i)}p_{\sigma(i)-1,N}(z_{\tau\sigma(i)})}{h_{\sigma(i)-1,N}}  \\
&=\frac{\prod_{n=0}^{N-1}h_{n,N}}{Z_N}\e^{-N\sum_{i=1}^NV(z_i)}\sum_{\tau\in S_N}\sign(\tau)\prod_{i=1}^N\left(\sum_{n=0}^{N-1}\frac{\overline{p_{n,N}(z_i)}p_{n,N}(z_{\tau(i)})}{h_{n,N}}\right) \\
&=\frac{\prod_{n=0}^{N-1}h_{n,N}}{Z_N}\det(K_N(z_i,z_j))_{i,j=1}^N.
\end{align*}

Using the properties of part (i), we find
\begin{align*}
&\int_D\det(K_N(z_i,z_j))_{i,j=1}^k\dd z_k \\
&\qquad=\det(K_N(z_i,z_j))_{i,j=1}^{k-1}\int_D K_N(z_k,z_k)\dd{z_k} \\
&\qquad\qquad-\sum_{j=1}^{k-1}\sum_{\sigma\in S_{k-1}}\sign(\sigma)\left(\prod_{\substack{i=1\\i\ne j}}^{k-1}K_N(z_i,z_{\sigma(i)})\right)\int_DK_N(z_j,z_k)K(z_k,z_{\sigma(j)})\dd{z_k} \\
&\qquad=(N-k+1)\det(K_N(z_i,z_j))_{i,j=1}^{k-1}.
\end{align*}
Integrating now recursively the expression for $P_N$, we get the relations of part~(ii). 
\end{proof}

\begin{proposition}
For $n\in\N$, the $n$-th orthogonal polynomial $p_{n,N}$ for the potential $V$ on the domain $D$ is given by
\begin{equation}\label{e:ogPolInt}
p_{n,N}(z_0)=\int_{D^n}\prod_{i=1}^n(z_0-z_i)P_{n,\frac NnV}(\vec z)\ddd{n}{\vec z},\quad z_0\in\C.
\end{equation}
\end{proposition}
\begin{proof}
By definition, the polynomial
\[ p_{n,N}(z_0)=(-1)^n\,\frac{\det\left(\left(z_0^i\right)_{i=0}^n,\left(\int_D\bar z_j^{j-1}z_j^i\e^{-NV(z_j)}\dd z_j\right)_{i=0,j=1}^n\right)}{\det\left(\int_D\bar z_j^{j-1}z_j^{i-1}\e^{-NV(z_j)}\dd z_j\right)_{i,j=1}^n}, \]
which by construction is orthogonal to all monomials of degree less than $n$ and has leading coefficient one, is the $n$-th orthogonal polynomial for the potential $V$. Pulling out the integrals of the determinant leads us to
\begin{align*}
p_{n,N}(z_0)&=(-1)^n\,\frac{\int_{D^n}\left(\prod_{j=1}^n\bar z_j^{j-1}\right)\det\left(z_j^i\right)_{i,j=0}^n\e^{-N\sum_{i=1}^nV(z_i)}\ddd n{\vec z}}{\int_{D^n}\left(\prod_{j=1}^n\bar z_j^{j-1}\right)\det\left(z_j^{i-1}\right)_{i,j=1}^n\e^{-N\sum_{i=1}^nV(z_i)}\ddd n{\vec z}} \\
&=(-1)^n\,\frac{\int_{D^n}\sum_{\sigma\in S_N}\sign(\sigma)\left(\prod_{j=1}^n\bar z_{\sigma(j)}^{j-1}\right)\det\left(z_j^i\right)_{i,j=0}^n\e^{-N\sum_{i=1}^nV(z_i)}\ddd n{\vec z}}{\int_{D^n}\sum_{\sigma\in S_N}\sign(\sigma)\left(\prod_{j=1}^n\bar z_{\sigma(j)}^{j-1}\right)\det\left(z_j^{i-1}\right)_{i,j=1}^n\e^{-N\sum_{i=1}^nV(z_i)}\ddd{n}{\vec z}}\\
&=\int_{D^n}\prod_{i=1}^n(z_0-z_i)P_{n,\frac NnV}(\vec z)\ddd{n}{\vec z}.\qedhere
\end{align*}
\end{proof}


\subsection{The potential $V(z)=\frac1{t_0}|z|^2$}
In this case, we can remove the cut-off by setting $D=\C$ and immediately find that the orthogonal polynomials are nothing but the monomials. So we get explicitly 
\[ \frac1NR_N^{(1)}(z)=\frac1{\pi t_0}\sum_{n=0}^{N-1}\frac{N^n}{t_0^n}\frac{|z|^{2n}}{n!}\e^{-\frac N{t_0}|z|^2}. \]
Considering now the limit $N\to\infty$, we get for the eigenvalue density
\begin{equation}
\begin{split}\label{e:densityCircle}
\lim_{N\to\infty}\frac1NR_N^{(1)}(z)&=\frac1{\pi t_0}\lim_{N\to\infty}\frac{N^{N+1}}{N!}\int_{\frac{|z|^2}{t_0}}^\infty x^{N-1}\e^{-Nx}\d x \\
&=\frac1{\pi t_0}\lim_{N\to\infty}\sqrt{\frac N{2\pi}}\int_{\frac{|z|^2}{t_0}}^\infty\frac1x\,\e^{-N(x-1-\log x)}\d x\\
&=\frac1{\pi t_0}\lim_{N\to\infty}\sqrt{\frac N{2\pi}}\int_{\frac{|z|^2}{t_0}-1}^\infty\frac1{1+x}\,\e^{-\frac N2x^2}\d x
=\begin{cases}0,&|z|^2>t_0,\\\frac1{2\pi t_0},&|z|^2=t_0,\\\frac1{\pi t_0},&|z|^2<t_0.\end{cases}
\end{split}
\end{equation}
In the continuum limit $N\to\infty$ therefore, the eigenvalues uniformly fill the disc with radius $\sqrt{t_0}$.

\begin{proposition}
Choosing the radius $\lambda=\sqrt{\frac xN}$, we find
\begin{equation}\label{e:A_N}
A_N(n,\lambda,0)=\left(\prod_{i=1}^N\Sigma_i(x)\right)\sum_{i_1<\ldots<i_n}\prod_{j=1}^n\frac{1-\Sigma_{i_j}(x)}{\Sigma_{i_j}(x)}, 
\end{equation}
where $\Sigma_i(x) = \e^{-\frac x{t_0}}\sum_{j=0}^{i-1}\frac{x^j}{j!\,t_0^j}$.
\end{proposition}
\begin{proof}
Since the orthonormal polynomials $(q_{n,N})_{n=0}^\infty$, $q_{n,N}(z)=\sqrt{\frac{N^{n+1}}{\pi t_0^{n+1}n!}}\,z^n$, for the potential $V$ on $\C$ are still orthogonal on $B_\lambda(0)$, namely
\[ \Int{\C\setminus B_\lambda(0)}{\overline{q_{i,N}(z)}q_{j,N}(z)\e^{-\frac N{t_0}|z|^2}\dd z} = \delta_{i,j}-\Int{B_\lambda(0)}{\overline{q_{i,N}(z)}q_{j,N}(z)\e^{-\frac N{t_0}|z|^2}\dd z} = \delta_{i,j}\Sigma_{i+1}(x), \]
we get
\begin{align*}
A_N(n,\lambda,0) &= \frac1{n!(N-n)!}\Int{B_\lambda(0)^n\times(\C\setminus B_\lambda(0))^{N-n}}{\e^{-\frac N{t_0}\sum_{i=1}^N|z_i|^2}\left|\det(q_{i-1}(z_j))_{i,j=1}^N\right|^2\ddd Nz} \\
&= \frac1{n!(N-n)!}\sum_{\sigma\in S_N}\prod_{i=1}^n(1-\Sigma_{\sigma(i)}(x))\prod_{i=n+1}^N\Sigma_{\sigma(i)}(x) \\
&= \left(\prod_{i=1}^N\Sigma_i(x)\right)\sum_{i_1<\ldots<i_n}\prod_{j=1}^n\frac{1-\Sigma_{i_j}(x)}{\Sigma_{i_j}(x)}.\qedhere
\end{align*}
\end{proof}

\begin{proposition}
The probability of finding exactly $n$ eigenvalues in a disc of radius $\lambda=\sqrt{\frac xN}$ around any point $a$ in the interior of the disc $B_{\sqrt{t_0}}$ is asymptotically for $N\to\infty$ given by $A_N(n,\lambda,0)$.
\end{proposition}
\begin{proof}
According to Proposition \ref{p:levelSpacing}, the probability $A_N(n,\lambda,a)$ is given by
\[ A_N(n,\lambda,a) = \frac1{n!}\sum_{k=0}^{N-n}\frac{(-1)^k}{k!}\Int{B_1(0)^{n+k}}{\lambda^{2(n+k)}R_N^{(n+k)}((a+\lambda z_i)_{i=1}^{n+k})\ddd{(n+k)}{\vec z},} \]
where the correlation functions $R^{(k)}_N$, $1\le k\le N$, can be determined by
\[ R^{(k)}_N\big((z_i)_{i=1}^k\big) = \det(K_N(z_i,z_j))_{i,j=1}^k \]
with
\[ K_N(w,z) = Nk_N(\bar wz)\e^{-\frac N{2t_0}(|z|^2+|w|^2-2\bar wz)},\quad k_N(z)=\frac1{\pi t_0}\sum_{m=0}^{N-1}\frac{N^m}{t_0^m}\frac{z^m}{m!}\e^{-\frac N{t_0}z}. \]
Now $k_N(|z|^2)=\frac1NR^{(1)}_N(z)$ for all $z\in\C$, and so calculation \eqref{e:densityCircle} tells us that, for $z$ in some domain containing the interval $(0,t_0)$, we locally uniformly have $k_N(z)=\frac1{\pi t_0}+\ord(\e^{-N\varepsilon})$ in the limit $N\to\infty$ for some $\varepsilon>0$. Therefore, there exists some $\varepsilon>0$ such that in the limit $N\to\infty$ locally uniformly in $w$ and $z$
\[ K_N(a+\lambda w, a+\lambda z)=N\left(\frac1{\pi t_0}+\ord(\e^{-N\varepsilon})\right)\e^{-\frac x{2 t_0}(|z|^2+|w|^2-2\bar wz)+\i\frac{\sqrt{xN}}{t_0}\IM(\bar a(z-w))}. \]
So, for $1\le k\le N$,
\[ \frac1{N^k}R^{(k)}_N\big((a+\lambda z_i)_{i=1}^k\big)=\left(\frac1{(\pi t_0)^k}+\ord(\e^{-N\frac\varepsilon2})\right)\sum_{\sigma\in S_k}\sign(\sigma)\e^{-\frac x{t_0}\sum_{i=1}^k\bar z_i(z_i-z_{\sigma(i)})} \]
is asymptotically independent of $a$ and so is the probability $A_N(n,\lambda,a)$.
\end{proof}


\section{Polynomial curves}
\subsection{Harmonic moments}
\begin{definition}
A polynomial curve of degree $d$ is a smooth simple closed curve in the complex plane with a parametrisation $h:S^1\subset\C\to\C$ of the form
\begin{equation}\label{e:polCurve}
h(w)= rw+\sum_{j=0}^d a_jw^{-j},\quad |w|=1,
\end{equation}
with $r>0$ and $a_d\ne0$. The standard (counterclockwise) orientation of the circle induces an orientation on the curve. We say that a polynomial curve is positively oriented if this orientation is counterclockwise, that is if the tangent vector to the curve makes one full turn in counterclockwise direction as we go around the unit circle.  
\end{definition}

\begin{proposition}\label{p:polCurve}
Let $\gamma$ be a positively oriented polynomial curve with parametrisation $h$ of the form \eqref{e:polCurve}. Then $h$, viewed as homolorphic map on $\C^\times$, restricts to a biholomorphic map from the exterior of the unit disc onto the exterior of $\gamma$.
\end{proposition}
\begin{proof}
We have to show that $h'(w)\ne0$ for all $w$ in the complement of the unit disc. Let $t$ denote the tangent vector map $w\mapsto t(w)=h'(w)\i w=\i(rw-\sum ja_jw^{-j})$. Since $\gamma$ is a simple closed curve, the map $w\mapsto t(w)/|t(w)|$ is a map of degree one from the unit circle to itself. Therefore, we have
\[ 1=\frac1{2\pi}\oint_{|w|=1}\!\!\!\!\!\!\!\d\,\mathrm{arg}(t(w)) =\frac1{2\pi\i}\oint_{|w|=1} \frac{t'(w)}{t(w)}\d w =N+\frac1{2\pi\i}\oint_{|w|=R} \frac{t'(w)}{t(w)}\d w. \]
Here $N\ge0$ denotes the number of zeros of $t(w)$, counted with multiplicity, in the complement of the unit disc and $R$ is so large that it contains them all. The latter integral is one as can be seen by sending $R$ to infinity. Thus, $N=0$, and $h'$ has no zeros in the complement of the unit disc.
\end{proof}

A simple consequence of this proposition is that a polynomial curve is uniquely para\-metrised by a map of the form \eqref{e:polCurve} with $r>0$. Indeed, any other conformal mapping of the complement differs by an automorphism of the complement of the unit disc. But non-trivial automorphisms are given by fractional linear transformations which do not preserve the conditions.

In the following, we therefore mean by the parametrisation of a polynomial curve always the parametrisation of the form \eqref{e:polCurve}.


\begin{definition}
The exterior harmonic moments $(t_k)_{k=1}^\infty$ of a polynomial curve $\gamma$ encircling the origin are defined by
\[ t_k= \frac1{2\pi\i k}\oint_\gamma\bar zz^{-k}\d z,\quad k\in\N. \]
The interior harmonic moments $(v_k)_{k=1}^\infty$ of $\gamma$ are 
\[ v_k =\frac1{2\pi\i}\oint_\gamma\bar zz^k\d z,\quad k\in\N. \]
\end{definition}

\begin{proposition}\label{p:polCurve/harmMom}
Let $\gamma$ be a positively oriented polynomial curve of degree $d$ encircling the origin.
\begin{enumerate}
\item 
The exterior harmonic moments $t_k$ of $\gamma$ vanish for all $k>d+1$.
\item
There exist universal polynomials $P_{j,k}\in\Z[r,a_0,\ldots,a_{j-k}]$, $1\le k\le j$, so that for $1\le k\le d+1$,
\begin{equation}\label{e:harmonicMoments}
kt_k=\bar a_{k-1}r^{-k+1}+ \sum_{j=k}^d\bar a_jr^{-j}P_{j,k}(r,a_0,\dots,a_{j-k}).
\end{equation}
Moreover, $P_{j,k}$ is a homogeneous polynomial of degree $j-k+1$ and it is also weighted homogeneous of degree $j-k+1$ for the assignment $\deg(a_j)=j+1$, $\deg(r)=0$.
\item
The area of the domain enclosed by $\gamma$ is $\pi t_0$ where
\begin{equation}\label{e:area}
t_0=r^2-\sum_{j=1}^dj|a_j|^2.
\end{equation}
\end{enumerate}
\end{proposition}

\begin{proof}
Let $h$ be the parametrisation of the polynomial curve $\gamma$. Then, since $\gamma$ encircles the origin, $h(w)\ne0$ for $|w|\ge1$. Hence, the contour in the formula for $t_k$ may be computed by taking residues at infinity. For $k\ge1$,
\begin{align*}
kt_k&=\frac1{2\pi\i}\oint_{|w|=1}{\bar h(w^{-1})}{h'(w)}h(w)^{-k}\d w\\
&=r^{-k}\sum_{j=0}^d\frac{\bar a_j}{2\pi\i}\oint_{|w|=1}w^{j-k}\left(r-\sum_{\ell=1}^d\ell a_\ell w^{-\ell-1}\right)\left(1+\sum_{\ell=0}^d\frac{a_\ell}r\,w^{-\ell-1}\right)^{-k}\d w.
\end{align*}
The integrals in this sum vanish if $j\le k-2$. The formula for $t_k$ in terms of $r$ and $(a_j)_{j=1}^d$ is obtained by expanding the geometric series and picking the coefficient of $w^{-1}$ in the integrand. This proves (i) and the first part of (ii). The homogeneity property is clear. The weighted homogeneity follows by rescaling $w$ in the integral. The same formula can be used to compute $t_0$, but the first term $rw^{-1}$ in $\bar h(w^{-1})$, which does not contribute to the integral and was omitted for $k\ge1$, must be added here.
\end{proof}


\begin{theorem}\label{t:uniquePolCurve} 
Let $t_2,\dots,t_{d+1}$ be complex numbers so that $|t_2|<\frac12$. Then there exists a $\delta>0$ so that for all $t_0$, $t_1$ with $0<t_0<\delta$ and $|t_1|^2<t_0(\frac12-|t_2|)$, there exists a unique positively oriented polynomial curve of degree less than or equal to $d$ encircling the origin with area $\pi t_0$ and exterior harmonic moments $(t_k)_{k=1}^\infty$ with $t_k=0$ for $k>d+1$.
\end{theorem}

\begin{proof}
The idea is to invert the map $(r,a_0,\dots,a_d)\mapsto (t_0,\dots,t_{d+1})$ defined by \eqref{e:harmonicMoments} and \eqref{e:area} for small $r$ and $a_0$. Set $\alpha_j=r^{-j}a_j$, $\rho=r^2$ and consider the resulting polynomial map 
\[ F:(\rho,\alpha_0,\dots,\alpha_d)\to (t_0,\dots,t_{d+1}), \] 
as a map from $\mathbb R\times\C^{n+1}$ to itself. The first claim is that this map has a smooth inverse in some neighbourhood of any point $(t_0,\dots,t_{d+1})\in\R\times\C^{d+1}$ with $t_0=t_1=0$ and $|t_2|\ne\frac12$.
By Proposition \ref{p:polCurve/harmMom}, this map is given by 
\begin{align*}
t_0&=\rho-\sum_{j=1}^d\rho^j j|\alpha_j|^2 \\
kt_k&=\bar\alpha_{k-1}+\sum_{j=k}^d\bar\alpha_jP_{j,k}(r,\alpha_0,r\alpha_1,\dots,r^{j-k}\alpha_{j-k}) \\
&=\bar\alpha_{k-1}+\sum_{j=k}^d\bar\alpha_jP_{j,k}(\rho,\alpha_0,\alpha_1,\dots,\alpha_{j-k}),\quad 1\le k\le d+1,
\end{align*}
where the universal polynomials $P_{j,k}$, $1\le k\le j$, are given by
\begin{align*}
&P_{j,k}(\rho,\alpha_0,\dots,\alpha_{j-k})=\\
&\qquad\frac1{2\pi\i}\oint_{|w|=R} w^{j-k}\left(1-\sum_{\ell=1}^d\ell\alpha_\ell\rho^\ell w^{-\ell-1}\right)\left(1+\sum_{\ell=0}^d\alpha_\ell\rho^\ell w^{-\ell-1}\right)^{-k}\d w
\end{align*}
for any sufficiently large $R$. By computing the residue at infinity, we can calculate $P_{j,k}$ and thus $t_k$ up to terms of at least second order in $\alpha_0$ and $\rho$,
\begin{align*}
t_0&=(1-|\alpha_1|^2)\rho+\cdots,\\
kt_k&=\bar\alpha_{k-1}-k\bar\alpha_k\alpha_0-(k+1)\bar\alpha_{k+1}\alpha_1\rho+\cdots,\quad1\le k\le d+1.
\end{align*}
Hence, $F(0,0,2\,\bar t_2,\dots,(d+1)\,\bar t_{d+1})=(0,0,t_2,\dots,t_{d+1})$ and the tangent map at this point sends $(\dot\rho,\dot\alpha_0,\dots,\dot\alpha_d)$ to $(\dot t_0,\dots,\dot t_{d+1})$ with
\begin{align*}
\dot t_0&=(1-4|t_2|^2)\dot\rho,\\
k\dot t_k&=\bar{\dot\alpha}_{k-1}-k(k+1)t_{k+1}\dot\alpha_0-2(k+1)(k+2)t_{k+2}\bar t_2\dot\rho,\quad 1\le k\le d+1.
\end{align*}
The tangent map is invertible if $|t_2|\ne\frac12$. By the inverse function theorem, $F$ has a smooth inverse on some neighbourhood of $(0,0,t_2,\dots,t_{d+1})$. If $|t_2|<\frac12$, $F$ preserves the positivity of the first coordinate.

In terms of the original variables, this means that given any $(t_0,\dots,t_{d+1})$ with small $t_0>0$, $t_1$ and such that  $|t_2|\ne\frac12$,  there is a curve $w\mapsto h(w)$  with
\[ h(w)=rw+\sum_{j=0}^dr^j\alpha_jw^{-j}\quad\textrm{and}\quad\alpha_j=(j+1)\,\bar t_{j+1}+\Ord(r^2). \]
There remains to show that if $r>0$ is small enough, $h$ parametrises a positively oriented simple closed curve containing the origin. We first show that $h$ is an immersion. Since $h'(w)=r-r\alpha_1w^{-2}+\Ord(r^2)$ and $\lim_{r\to 0}\alpha_1=2\,\bar t_2$, we see that as long as $|t_2|\ne\frac12$, $h'(w)$ does not vanish on the unit circle. Similarly, we show that $h:S^1\to\C$ is injective: we have
\begin{align*}
|h(w)-h(w')|^2&=r|w-w'+2\,\bar t_2(w^{-1}-{w'}^{-1})|+\Ord(r^2) \\
&=r|w-w'+2\,\bar t_2(\bar w-\bar w')|+\Ord(r^2).
\end{align*}
But the expression in the absolute value can only vanish for $w\ne w'$ if $|2\,\bar t_2|=1$, which is excluded by the hypothesis. Moreover, $h(w)=rw+\bar t_1+2r\bar t_2w^{-1}+\Ord(r^2)$. Therefore, $h$ parametrises a perturbation of an ellipse centered at $\bar t_1$. The condition on $t_1$ is a sufficient condition for this ellipse to contain the origin. 
\end{proof}

\begin{definition}
We denote the space of all harmonic moments $\vec t=(t_k)_{k=0}^\infty$ (including the area $t_0$ of the interior domain in units of $\pi$) fulfilling the conditions of Theorem \ref{t:uniquePolCurve} for $d\in\N$ by $\mathcal T_d$.
\end{definition}

In the following, we will always consider a polynomial curve to be positively oriented.



\subsection{The Schwarz function}
\begin{definition}
The Schwarz function of a polynomial curve $\gamma$ is defined as the analytic continuation (in a neighbourhood of the curve) of the function $S(z)=\bar z$ on $\gamma$. The Schwarz reflection $\rho$ for the polynomial curve in this domain is the anti-holomorphic map $\rho(z)=\overline{S(z)}$.
\end{definition}

\begin{definition}
Let $\gamma$ be a polynomial curve and $h$ its parametrisation. The critical radius $R$ of $\gamma$ is then defined as
\[ R = \max\{|w|\;|\;h'(w)=0,\; w\in\C^\times\}, \]
which by Proposition \ref{p:polCurve} is less than one.
\end{definition}

The following proposition now shows that the Schwarz function is indeed well defined.
\begin{proposition}
Let $\gamma$ be a polynomial curve parametrised by $h$ and $R$ its critical radius. Then the Schwarz function $S$ and the Schwarz reflection $\rho$ of the curve $\gamma$ restricted to $h(B_{R\I}\setminus\bar B_R)$, where $B_R$ denotes the open disc with radius $R$ around zero, are biholomorphic and anti-biholomorphic maps respectively. They are given by
\begin{equation}
S(z) = \bar h\left(\frac{1}{h\I(z)}\right)\quad\textrm{and}\quad\rho(z)=h\left(\frac{1}{\bar h\I(\bar z)}\right),\label{e-reflection}
\end{equation}
where $h\I$ denotes the inverse of $h|_{\C\setminus B_R}$.

Therefore, $\rho$ maps $\gamma$ identically on itself, $h(B_1\setminus\bar B_R)$ to $h(B_{R\I}\setminus\bar B_1)$ and vice versa. Additionally, we have that $\rho\circ\rho$ is the identity.
\end{proposition}
\begin{proof}
By definition of the Schwarz function, in a neighbourhood of $|w|=1$,
\[ S(h(w))=\bar h(w\I). \]
Because for $w\in\C\setminus\bar B_R$ the function $h$ is biholomorphic, we may write
\[ S(z) = \bar h\left(\frac{1}{h\I(z)}\right),\quad \rho(z)=\overline{S(z)}=h\left(\frac{1}{\bar h\I(\bar z)}\right). \]
Taking the derivatives, we find that they do not vanish for $z\in h(B_{R\I}\setminus\bar B_R)$, and so $S$ and $\rho$ are biholomorphic functions on $h(B_{R\I}\setminus\bar B_R)$.
\end{proof}
\begin{proposition}\label{p:expansion of schwarz function}
The Schwarz function of a polynomial curve $\gamma$ of degree $d$ encircling the origin, with parametrisation $h$ and critical radius $R$, is analytic on $h(\C\setminus\bar B_R)$ and has the Laurent representation
\begin{equation}\label{e:SchwarzFunctionExpansion}
S(z)=\sum_{k=1}^{d+1}kt_kz^{k-1}+\frac{t_0}z+\sum_{k=1}^\infty v_kz^{-k-1}
\end{equation}
around infinity, where $(t_k)_{k=1}^\infty$ are the exterior harmonic moments of $\gamma$, $(v_k)_{k=1}^\infty$ are the interior harmonic moments of $\gamma$, and $\pi t_0$ denotes the area of the interior domain of $\gamma$.
\end{proposition}
\begin{proof}
With Cauchy's integral formula, this follows immediately from the definition of the harmonic moments:
\begin{align*}
&t_0=\frac1{2\pi\i}\oint_\gamma S(z)\d z,\quad t_k = \frac1{2\pi\i k}\oint_\gamma z^{-k}S(z)\d z,\quad\textrm{and}\\
&v_k = \frac1{2\pi\i}\oint_\gamma z^kS(z)\d z\quad\text{for}\quad k\in\N. \qedhere
\end{align*}
\end{proof}

\begin{proposition}\label{p:singularitiesSchwarz}
Inside a polynomial curve $\gamma$ the Schwarz function $S$ of $\gamma$ becomes a multivalued function which has no singularities, unless the curve is a circle where we have $S(z)=\frac{t_0}z$.
\end{proposition}
\begin{proof}
Because of the representation \eqref{e-reflection} for the Schwarz function, a branch of the Schwarz function may only diverge in $z\in\C$ if $z\in\{h(0),h(\infty)\}$. If the polynomial curve is not a circle, we have $h(0)=\infty=h(\infty)$. So, the only singularity is located at infinity. 
\end{proof}


\subsection{The Toda lattice hierarchy}
\begin{definition}
Let $N\in\N$. The Toda lattice hierarchy is defined by the Lax-Sato equations
\begin{align}
&\frac{\partial L_N}{\partial t_k}=\frac N{t_0}[M^{(k)}_N,L_N],&&\frac{\partial\tilde L_N}{\partial t_k}=\frac N{t_0}[M^{(k)}_N,\tilde L_N], \label{e:Toda1} \\
&\frac{\partial L_N}{\partial \tilde t_k}=\frac N{t_0}[L_N,\tilde M^{(k)}_N],&&\frac{\partial\tilde L_N}{\partial \tilde t_k}=\frac N{t_0}[\tilde L_N,\tilde M^{(k)}_N], \label{e:Toda2}
\end{align}
$k\in\N$, for the difference operators
\begin{align*}
&L_N(t_0,\vec t,\vec{\tilde t})=r_N(t_0,\vec t,\vec{\tilde t})\e^{\frac{t_0}N\partial_{t_0}}+\sum_{j=0}^\infty a_{j,N}(t_0,\vec t,\vec{\tilde t})\e^{-\frac{jt_0}N\partial_{t_0}}\quad\textrm{and}\\
&\tilde L_N(t_0,\vec t,\vec{\tilde t})=\tilde r_N(t_0,\vec t,\vec{\tilde t})\e^{-\frac{t_0}N\partial_{t_0}}+\sum_{j=0}^\infty\tilde a_{j,N}(t_0,\vec t,\vec{\tilde t})\e^{\frac{jt_0}N\partial_{t_0}},
\end{align*}
$\vec t=(t_k)_{k=1}^\infty$, $\vec{\tilde t}=(\tilde t_k)_{k=1}^\infty$, on the space of all analytic functions in $t_0$, where
\begin{equation}\label{e:MMatrix}
M^{(k)}_N=(L_N^k)_++\frac12(L_N^k)_0\quad\textrm{and}\quad \tilde M^{(k)}_N=(\tilde L_N^k)_-+\frac12(\tilde L_N^k)_0,\quad k\in\N.
\end{equation}
Here $O_+$, $O_0$, and $O_-$ denote the positive, constant and negative part of the operator $O$ in the shift operator $\e^{\frac{t_0}N\partial_{t_0}}$.
\end{definition}

\newcommand{\M}[1]{M_N^{(#1)}}
\renewcommand{\Mc}[1]{M_N^{\mathrm{c}\,(#1)}}
\newcommand{\tM}[1]{\tilde M_N^{(#1)}}
\newcommand{\tMc}[1]{\tilde M_N^{\mathrm{c}\,(#1)}}
\begin{proposition}\label{p:compatibility}
Let $L_N$ and $\tilde L_N$ be a solution of the Toda lattice hierarchy. Then, with $\M k$ and $\tM k$ given by \eqref{e:MMatrix}, the compatibility relations
\begin{align}
&\frac{\partial\M\ell}{\partial t_k}-\frac{\partial\M k}{\partial t_\ell}=\frac N{t_0}[\M k,\M\ell],  
&&\frac{\partial\tM\ell}{\partial\tilde t_k}-\frac{\partial\tM k}{\partial\tilde t_\ell}=\frac N{t_0}[\tM\ell,\tM k], \label{e:compatibilityToda1} \\
&\frac{\partial\M\ell}{\partial\tilde t_k}+\frac{\partial\tM k}{\partial t_\ell}=\frac N{t_0}[\M\ell,\tM k], \label{e:compatibilityToda2}
\end{align}
$k,\ell\in\N$, for the equations \eqref{e:Toda1} and \eqref{e:Toda2} are fulfilled.
\end{proposition}
\begin{proof}
Let $k,\ell\in\N$, and let us introduce $\Mc k=L_N^k-\M k$ and $\tMc k=\tilde L_N^k-\tM k$.
Then 
\begin{align*}
\frac{\partial L_N^\ell}{\partial t_k}-\frac{\partial L_N^k}{\partial t_\ell} &= \frac N{t_0}[\M k,\M\ell+\Mc\ell]+\frac N{t_0}[\Mc\ell,\M k+\Mc k] \\
&=\frac N{t_0}[\M k,\M\ell]+\frac N{t_0}[\Mc\ell,\Mc k].
\end{align*}
Adding now the positive and half of the zeroth part of this equation gives us the first of the relations \eqref{e:compatibilityToda1}.
On the same way, we get by taking the negative and zeroth part of
\begin{align*}
\frac{\partial\tilde L_N^\ell}{\partial\tilde t_k}-\frac{\partial\tilde L_N^k}{\partial\tilde t_\ell} &= \frac N{t_0}[\tM\ell+\tMc\ell,\tM k]+\frac N{t_0}[\tM k+\tMc k,\tMc\ell] \\
&=\frac N{t_0}[\tM\ell,\tM k]+\frac N{t_0}[\tMc k,\tMc\ell].
\end{align*}
the second relation of \eqref{e:compatibilityToda1}.
Now
\begin{align*}
&\left(\frac{\partial\M\ell}{\partial\tilde t_k}+\frac{\partial\tM k}{\partial t_\ell}-\frac N{t_0}[\M\ell,\tM k]\right)_+=\left(-\frac{\partial\Mc\ell}{\partial\tilde t_k}+\frac{\partial\tM k}{\partial t_\ell}+\frac N{t_0}[\Mc\ell,\tM k]\right)_+ = 0, \\
&\left(\frac{\partial\M\ell}{\partial\tilde t_k}+\frac{\partial\tM k}{\partial t_\ell}-\frac N{t_0}[\M\ell,\tM k]\right)_-=\left(\frac{\partial\M\ell}{\partial\tilde t_k}-\frac{\partial\tMc k}{\partial t_\ell}+\frac N{t_0}[\M\ell,\tMc k]\right)_- = 0, 
\end{align*}
and therefore, since
\begin{align*}
\left(-\frac{\partial\Mc\ell}{\partial\tilde t_k}+\frac{\partial\tM k}{\partial t_\ell}+\frac N{t_0}[\Mc\ell,\tM k]\right)_0 
&= -\frac12\left(\frac{\partial L_N^\ell}{\partial\tilde t_k}\right)_0+\frac12\left(\frac{\partial\tilde L_N^k}{\partial t_k}\right)_0 \\
&=-\left(\frac{\partial\M\ell}{\partial\tilde t_k}-\frac{\partial\tMc k}{\partial t_\ell}+\frac N{t_0}[\M\ell,\tMc k]\right)_0,
\end{align*}
equation \eqref{e:compatibilityToda2} is fulfilled, too. 
\end{proof}

Taking formally the limit $N\to\infty$ of the Toda lattice hierarchy, the shift operator $\e^{\frac{t_0}N\partial_{t_0}}$ will be replaced by a variable $w$ and the scaled commutator $\frac N{t_0}[\cdot,\cdot]$ becomes a Poisson bracket with respect to the canonical variables $\log w$ and $t_0$. This leads us to the following definition of the dispersionless Toda lattice hierarchy.
\begin{definition}\label{d:dispersionlessTH}
The dispersionless Toda lattice hierarchy is given by the equation system
\begin{align*}
&\frac{\partial z}{\partial t_k}=\{M_k,z\},&&\frac{\partial\tilde z}{\partial t_k}=\{M_k,\tilde z\}, \\
&\frac{\partial z}{\partial\tilde t_k}=\{z,\tilde M_k\},&&\frac{\partial\tilde z}{\partial\tilde t_k}=\{\tilde z,\tilde M_k\},
\end{align*}
$k\in\N$, for functions $z$ and $\tilde z$ of $w$, $t_0$, $\vec t=(t_k)_{k=1}^\infty$, and $\vec{\tilde t}=(\tilde t_k)_{k=1}^\infty$ of the form
\begin{align*}
&z(w,t_0,\vec t,\vec{\tilde t})=r(t_0,\vec t,\vec{\tilde t})w+\sum_{j=0}^\infty a_j(t_0,\vec t,\vec{\tilde t})w^{-j}\quad\textrm{and} \\
&\tilde z(w,t_0,\vec t,\vec{\tilde t})=\tilde r(t_0,\vec t,\vec{\tilde t})w^{-1}+\sum_{j=0}^\infty\tilde a_j(t_0,\vec t,\vec{\tilde t})w^j.
\end{align*}
Here, with $f_+$, $f_0$, and $f_-$ denoting the positive, constant and negative part of a function $f$ considered as power series in $w$, 
\begin{equation}\label{e:dispersionlessMMatrix}
M_k=(z^k)_++\frac12(z^k)_0\quad\textrm{and}\quad \tilde M_k=(\tilde z^k)_-+\frac12(\tilde z^k)_0,\quad k\in\N,
\end{equation}
and the Poisson bracket is defined as
\[ \{f,g\}=w\frac{\partial f}{\partial w}\frac{\partial g}{\partial t_0}-w\frac{\partial f}{\partial t_0}\frac{\partial g}{\partial w}. \]
\end{definition}
The compatibility relations for the dispersionless Toda lattice hierarchy are also given as the limit $N\to\infty$ of the compatibility relations \eqref{e:compatibilityToda1} and \eqref{e:compatibilityToda2} of the dispersionful Toda lattice hierarchy.
\begin{proposition}
Let $z$ and $\tilde z$ be a solution of the dispersionless Toda lattice hierarchy. Then $M_k$ and $\tilde M_k$ given by \eqref{e:dispersionlessMMatrix} fulfil the equations
\begin{align*}
&\frac{\partial M_\ell}{\partial t_k}-\frac{\partial M_k}{\partial t_\ell}=\{M_k,M_\ell\}, &&\frac{\partial\tilde M_\ell}{\partial\tilde t_k}-\frac{\partial\tilde M_k}{\partial\tilde t_\ell}=\{\tilde M_\ell,\tilde M_k\}, \\
&\frac{\partial M_\ell}{\partial\tilde t_k}+\frac{\partial\tilde M_k}{\partial t_\ell}=\{M_\ell,\tilde M_k\}, &&k,\ell\in\N.
\end{align*}
\end{proposition}
\begin{proof}
The proof follows exactly the lines of the proof of Proposition \ref{p:compatibility}. We only have to adapt the notation and substitute the scaled commutators $\frac N{t_0}[\cdot,\cdot]$ by the Poisson brackets $\{\cdot,\cdot\}$.
\end{proof}

\subsection{Integrable hierarchy of conformal maps}
We consider the parametrisation $h$ of a polynomial curve as function of the parameter $w\in\C^\times$, of the area $t_0$ (in units of $\pi$) and the harmonic moments $(t_k)_{k=1}^\infty$ of the curve. Set $z(w,\vec t)=h(w)$ and $\tilde z(w,\vec t)=\bar h(w\I)$, where $\vec t=(t_k)_{k=0}^\infty\in\mathcal T_d$ for some $d\in\N$ and $w\in\C^\times$.
\begin{proposition}\label{p:canonical transformation}
On $\C^\times\times\mathcal T_d$, we have 
$\{z,\tilde z\}=1$,
where the Poisson bracket $\{\cdot,\cdot\}$ is defined as
\begin{equation}\label{e:PoissonBracket}
\{f,g\}(w,\vec t)=w\frac{\partial f}{\partial w}(w,\vec t)\frac{\partial g}{\partial t_0}(w,\vec t)-w\frac{\partial f}{\partial t_0}(w,\vec t)\frac{\partial g}{\partial w}(w,\vec t).
\end{equation}
\end{proposition}
\begin{proof}
Let $S(\cdot,\vec t)$ be the Schwarz function of the polynomial curve defined by the set $\vec t\in\mathcal T_d$ of harmonic moments. Then, for $|w|>1$, $\tilde z(w,\vec t)=S(z(w,\vec t),\vec t)$, and we get
\begin{align*}
\{z,\tilde z\}(w,\vec t)&=w\frac{\partial z}{\partial w}(w,\vec t)\left(\frac{\partial S}{\partial t_0}(z(w,\vec t),\vec t)+\frac{\partial S}{\partial z}(z(w,\vec t),\vec t)\frac{\partial z}{\partial t_0}(w,\vec t)\right) \\
&\qquad-w\frac{\partial z}{\partial t_0}(w,\vec t)\frac{\partial S}{\partial z}(z(w,\vec t),\vec t)\frac{\partial z}{\partial w}(w,\vec t)\\ 
&= w\frac{\partial z}{\partial w}(w,\vec t)\frac{\partial S}{\partial t_0}(z(w,\vec t),\vec t),
\end{align*}
whose Laurent series in $w$ around infinity has the form $1+\Ord(w\I)$, as can be seen from Proposition \ref{p:expansion of schwarz function}.

On the other hand, using $z(w,\vec t)=\bar S(\tilde z(w,\vec t),\vec t)$ for $|w|>1$, where we defined $\bar S(z,\vec t)=\overline{S(\bar z,\vec t)}$, we get
\[ \{z,\tilde z\}(w,\vec t)
=-w\frac{\partial\tilde z}{\partial w}(w,\vec t)\frac{\partial\bar S}{\partial t_0}(\tilde z(w,\vec t),\vec t), \]
which has a Laurent series of the form $1+\Ord(w)$.

So, by analytic continuation in $w$, we have $\{z,\tilde z\}=1$ on $\C^\times\times\mathcal T_d$.
\end{proof}

\begin{proposition}\label{p:generatingFunction}
There exists a function $\Omega:\{(z,\vec t)\in\C\times\mathcal T_d\;|\;z\in(D_{\vec t})_-\}\to\C$ with
\begin{equation}\label{e:omegaDef}
S(z(w,\vec t),\vec t)=\frac{\partial\Omega}{\partial z}(z(w,\vec t),\vec t)\quad\textrm{and}\quad\log w=\frac{\partial\Omega}{\partial t_0}(z(w,\vec t),\vec t)\quad\textrm{for}\quad|w|>1,
\end{equation}
where $S(\cdot,\vec t)$ denotes the Schwarz function of the curve $\gamma_{\vec t}$ defined by the harmonic moments $\vec t\in\mathcal T_d$ and $(D_{\vec t})_-$ is the exterior domain of the curve $\gamma_{\vec t}$.

Additionally, every such function has an asymptotic expansion of the form
\begin{equation}\label{e:omegaExpansion}
\Omega(z,\vec t)=\sum_{k=1}^{d+1}t_kz^k+t_0\log z-\frac12v_0(\vec t)-\sum_{k=1}^\infty\frac{v_k(\vec t)}kz^{-k}
\end{equation}
around $z=\infty$, where $(v_k(\vec t))_{k=1}^\infty$ are the interior harmonic moments of $\gamma_{\vec t}$ and, with $z(w,\vec t)=r(\vec t)w+\Ord(1)$ for $w\to\infty$,
\begin{equation}\label{e:differentialv0}
\frac{\partial v_0}{\partial t_0}(\vec t)=2\log(r(\vec t)).
\end{equation}
\end{proposition}
\begin{proof}
For $|w|>1$ and $\vec t\in\mathcal T_d$, we write $\tilde z(w,\vec t)=S(z(w,\vec t),\vec t)$ and get with Proposition \ref{p:canonical transformation} the compatibility relation
\begin{align*}
\frac{\partial S}{\partial t_0}(z(w,\vec t),\vec t)&=\frac{\partial\tilde z}{\partial t_0}(w,\vec t)-\frac{\partial S}{\partial z}(z(w,\vec t),\vec t)\frac{\partial z}{\partial t_0}(w,\vec t) \\
&=\frac{\partial\tilde z}{\partial t_0}(w,\vec t)-\frac{\frac{\partial\tilde z}{\partial w}(w,\vec t)}{\frac{\partial z}{\partial w}(w,\vec t)}\,\frac{\partial z}{\partial t_0}(w,\vec t)
=\frac1{w\frac{\partial z}{\partial w}(w,\vec t)}
\end{align*}
for the equation system \eqref{e:omegaDef}, which implies the existence of such a function $\Omega$.

For the asymptotic expansion of $\Omega$, we integrate the expansion \eqref{e:SchwarzFunctionExpansion} of the Schwarz function $S$ with respect to $z$ and get with the integration constant $-\frac12v_0(\vec t)$ the equation \eqref{e:omegaExpansion}. Equation \eqref{e:differentialv0} is now implied by the zeroth order in $w$ of the relation
\[ \log w = \frac{\partial\Omega}{\partial t_0}(z(w,\vec t),\vec t) = \log w+\log(r(\vec t))-\frac12\frac{\partial v_0}{\partial t_0}(\vec t)+\Ord(w^{-1}). \qedhere \]
\end{proof}

\begin{proposition}
On $\C^\times\times\mathcal T_d$, we have for $1\le k\le d+1$
\begin{equation}\label{e:toda hierarchy}
\frac{\partial z}{\partial t_k}=\{M_k,z\},\quad\frac{\partial\tilde z}{\partial t_k}=\{M_k,\tilde z\},\quad\frac{\partial z}{\partial\bar t_k}=\{z,\tilde M_k\},\quad \frac{\partial\tilde z}{\partial\bar t_k}=\{\tilde z,\tilde M_k\},
\end{equation}
where the Poisson bracket $\{\cdot,\cdot\}$ is defined by \eqref{e:PoissonBracket} and, with the notation $f_+$, $f_-$ and $f_0$ for the positive, the negative, and the constant part of the Laurent series of $f$ in $w$,
\begin{equation}\label{e:todaHamiltonians}
M_k(w,\vec t)=(z^k(w,\vec t))_++\frac12(z^k(w,\vec t))_0,\quad\tilde M_k(w,\vec t)=(\tilde z^k(w,\vec t))_-+\frac12(\tilde z^k(w,\vec t))_0.
\end{equation}
\end{proposition}
\begin{proof}
We choose a function $\Omega$ as in Proposition \ref{p:generatingFunction} and define for $|w|>1$ and $\vec t\in\mathcal T_d$
\[ M_k(w,\vec t)=\frac{\partial\Omega}{\partial t_k}(z(w,\vec t),\vec t). \]
Considering now only $|w|>1$, we get, using $\frac{\partial\Omega}{\partial t_0}(z(w,\vec t),\vec t)=\log w$,
\begin{align*}
\{M_k,z\}(w,\vec t)&=-w\frac{\partial z}{\partial w}(w,\vec t)\frac{\partial^2\Omega}{\partial t_k\partial t_0}(z(w,\vec t),\vec t) \\
&=-w\frac{\partial z}{\partial w}(w,\vec t)\left(\frac\partial{\partial t_k}\left(\frac{\partial\Omega}{\partial t_0}(z(w,\vec t),\vec t)\right)-\frac{\partial^2\Omega}{\partial z\partial t_0}(z(w,\vec t),\vec t)\frac{\partial z}{\partial t_k}(w,\vec t)\right) \\
&= \frac{\partial z}{\partial t_k}(w,\vec t).
\end{align*}
Similarly, we have, using additionally $\tilde z(w,\vec t)=\frac{\partial\Omega}{\partial z}(z(w,\vec t),\vec t)$ and $\{z,\tilde z\}=1$,
\begin{align*}
\{M_k,\tilde z\}(w,\vec t)&=\frac{\partial^2\Omega}{\partial z\partial t_k}(z(w,\vec t),\vec t)\,\{z,\tilde z\}(w,\vec t)-w\frac{\partial\tilde z}{\partial w}(w,\vec t)\frac{\partial^2\Omega}{\partial t_0\partial t_k}(z(w,\vec t),\vec t) \\
&=\frac\partial{\partial t_k}\left(\frac{\partial\Omega}{\partial z}(z(w,\vec t),\vec t)\right)-\frac{\partial z}{\partial t_k}(w,\vec t)\frac{\partial^2\Omega}{\partial z^2}(z(w,\vec t),\vec t) \\
&\qquad-w\frac{\partial\tilde z}{\partial w}(w,\vec t)\left(\frac\partial{\partial t_k}\left(\frac{\partial\Omega}{\partial t_0}(z(w,\vec t),\vec t)\right)-\frac{\partial z}{\partial t_k}(w,\vec t)\frac{\partial^2\Omega}{\partial z\partial t_0}(z(w,\vec t),\vec t)\right) \\
&=\frac{\partial\tilde z}{\partial t_k}(w,\vec t).
\end{align*}
Solving these two equations for $\frac{\partial M_k}{\partial w}$, we find, again using $\{z,\tilde z\}=1$, 
\[ \frac{\partial M_k}{\partial w}=\frac{\partial z}{\partial w}\frac{\partial\tilde z}{\partial t_k}-\frac{\partial z}{\partial t_k}\frac{\partial\tilde z}{\partial w}. \]
With $z(w,\vec t)=\bar S(\tilde z(w,\vec t),\vec t)$, where as before $\bar S(z,\vec t)=\overline{S(\bar z,\vec t)}$, we then have
\[ \left(\frac{\partial M_k}{\partial w}(w,\vec t)\right)_-=\left(-\frac{\partial\tilde z}{\partial w}(w,\vec t)\frac{\partial\bar S}{\partial t_k}(\tilde z(w,\vec t),\vec t)\right)_-=0. \]
Therefore, substituting the asymptotic expansion \eqref{e:omegaExpansion} for $\Omega$,
\[ M_k(w,\vec t)=(z^k(w,\vec t))_++(z^k(w,\vec t))_0-\frac12\frac{\partial v_0}{\partial t_k}(\vec t). \] 
Considering now the derivative by $t_0$, we find as before
\begin{align*}
\frac{\partial M_k}{\partial t_0}(w,\vec t)&=\frac{\partial z}{\partial t_0}(w,\vec t)\frac{\partial\tilde z}{\partial t_k}(w,\vec t)-\frac{\partial z}{\partial t_k}(w,\vec t)\frac{\partial\tilde z}{\partial t_0}(w,\vec t) \\
&=\frac{\partial\bar S}{\partial t_0}(\tilde z(w,\vec t),\vec t)\frac{\partial\tilde z}{\partial t_k}(w,\vec t)-\frac{\partial\bar S}{\partial t_k}(\tilde z(w,\vec t),\vec t)\frac{\partial\tilde z}{\partial t_0}(w,\vec t).
\end{align*}
Writing $\tilde z(w,\vec t)=r(\vec t)w^{-1}+\Ord(1)$ for $w\to0$, we see
\[ \frac{\partial M_k}{\partial t_0}(w,\vec t)=\frac1{r(\vec t)}\frac{\partial r}{\partial t_k}(\vec t)+\Ord(w)=\frac12\frac{\partial^2v_0}{\partial t_0\partial t_k}(\vec t)+\Ord(w). \]
Therefore, using the freedom of choosing the $t_0$-independent part of $v_0$, we achieve
\[ M_k(w,\vec t)=(z^k(w,\vec t))_++\frac12(z^k(w,\vec t))_0. \]
By analytic continuation in $w$, we see that $M_k$, defined by \eqref{e:todaHamiltonians}, fulfils the relations \eqref{e:toda hierarchy} for all $w\in\C^\times$ and $\vec t\in\mathcal T_d$.

Setting now 
\[ \tilde M_k(w,\vec t) = \overline{M_k(\bar w\I,\vec t)}, \]
and using $\tilde z(w,\vec t) = \overline{z(\bar w\I,\vec t)}$, we get the corresponding relations for $\tilde M_k$.
\end{proof}

Therefore, the functions $z$ and $\tilde z$ are a solution of the dispersionless Toda lattice hierarchy in the first $d+1$ time variables $(t_k)_{k=1}^{d+1}$ fulfilling additionally the so called dispersionless string equation $\{z,\tilde z\}=1$.    

Furthermore, because of the additional restrictions $\tilde t_k=\bar t_k$, $1\le k\le d+1$, and $\tilde z(w,\vec t) = \overline{z(\bar w\I,\vec t)}$, $(w,\vec t)\in\C^\times\times\mathcal T_d$, compared to Definition \ref{d:dispersionlessTH} of the dispersionless Toda lattice hierarchy, the equations for the time variables $(\tilde t_k)_{k=1}^{d+1}$ are equivalent to those for the variables $(t_k)_{k=1}^{d+1}$.



\section{The eigenvalue density in the continuum limit}
\subsection{The equilibrium measure}
We now turn our attention to the limit $N\to\infty$.
To describe a configuration in this limit, it is useful to pass from the sequence of eigenvalues to measures. So let us introduce the point measures
\[ \delta_z(A)=\chi_A(z),\;z\in\C,\,A\subset\C,\quad\textrm{and}\quad \delta_{\vec z} = \frac1N\sum_{i=1}^N\delta_{z_i},\;\vec z=(z_i)_{i=1}^N\subset\C, \]
where $\chi_A$ denotes the characteristic function of the set $A\subset\C$. 
Then, in the domain $D_0^N$, the probability distribution $P_N$ may be written as
\[ P_N(\vec z)=\frac1{Z_N}\e^{-N^2I_V(\delta_{\vec z})}, \]
where we introduced the functional
\[ I_V(\mu)=\int V(z)\d\mu(z)+\iint_{z\ne\zeta}\log|z-\zeta|\I\d\mu(\zeta)\d\mu(z) \]
on the space of all Borel probability measures. (Again, we will drop the index indicating which potential we use.)

\begin{lemma}\label{l:compactness}
The space of all Borel probability measures on $D$ is sequentially compact.
\end{lemma}
\begin{proof}
By the theorem of Riesz-Markov each Borel measure $\mu$ on $D$ corresponds to exactly one positive, linear functional $\phi_\mu\in C(D)^*$ and by the theorem of Alaoglu, the closed unit-sphere in $C(D)^*$ is weak-*-compact. Therefore, for each sequence $(\mu_n)_{n=1}^\infty$ of Borel probability measures, the sequence $(\phi_{\mu_n})_{n=1}^\infty$ contains a weak-*-convergent subsequence $(\phi_{\mu_{n(k)}})_{k=1}^\infty$, that is
\[ \exists\,\phi\in C(D)^*\,:\;\phi_{\mu_{n(k)}}(f) \to \phi(f)\quad\forall f\in C(D). \]
Now we find a measure $\mu$ on $D$ with $\phi=\phi_\mu$. This measure fulfils
\[ \int f\d\mu_{n(k)}\to\int f\d\mu\quad(k\to\infty)\quad\forall f\in C(D), \]
and hence, is again a Borel probability measure.
\end{proof}

\begin{lemma}\label{l-integrability}
Let $\mu$ and $\tilde\mu$ be Borel probability measures such that the function $\log|z-\zeta|\I$ is integrable with respect to $\mu\otimes\mu$ and to $\tilde\mu\otimes\tilde\mu$. Then $\log|z-\zeta|\I$ is also integrable with respect to $\mu\otimes\tilde\mu$.

Additionally, we have the inequality
\begin{equation}\label{e-log}
\iint\log|z-\zeta|\I\d(\tilde\mu-\mu)(\zeta)\d(\tilde\mu-\mu)(z)\ge 0
\end{equation}
with equality if and only if $\mu=\tilde\mu$.
\end{lemma}
\begin{proof}
We start with the distributional identity
\[ \int_\C\log|z|\I\triangle\varphi(z)\dd z = -2\pi\varphi(0) \]
for any smooth function $\varphi:\C\to\R$ with compact support.
Introducing the Fourier transform 
\[ \hat\varphi(k)=\frac1{2\pi}\int_\C\varphi(z)\e^{\frac\i 2(kz+\bar k\bar z)}\dd z,\quad k\in\C, \]
of $\varphi$, we find
\begin{align*}
\int_\C\log|z|\I\triangle\varphi(z)\dd z &= -\int_\C\frac{1}{|k|^2}(|k|^2\hat\varphi(k))\dd k \\
&= \frac1{2\pi}\int_\C\frac{1}{|k|^2}\int_\C\triangle\varphi(z)\left(\e^{\frac{\i}{2}(kz+\bar k\bar z)}-f(k)\right)\!\dd z\dd k \\
&= \frac1{2\pi}\iint_{\C^2}\frac{1}{|k|^2}\left(\e^{\frac{\i}{2}(kz+\bar k\bar z)}-f(k)\right)\!\dd k\,\triangle\varphi(z)\dd z,
\end{align*}
where $f:\C\to[0,1]$ denotes a continuous function which is one in the vicinity of zero and becomes zero at infinity. So, we have for all $z\in\C^\times$ the equation
\[ \log|z|\I = \frac{1}{2\pi}\int_\C\frac{1}{|k|^2}\left(\e^{\frac{\i}{2}(kz+\bar k\bar z)}-f(k)\right)\!\dd k+C(f) \]
for some real constant $C(f)$.

Therefore, with Tonelli's theorem, we see that
\[ \iint\log|z-\zeta|\I\d(\tilde\mu-\mu)(\zeta)\d(\tilde\mu-\mu)(z)=\frac1{2\pi}\int_\C\frac1{|k|^2}\left|\int\e^{\frac{\i}{2}(kz+\bar k\bar z)}\d(\tilde\mu-\mu)(z)\right|^2\dd k \]
is non-negative and finite, which immediately implies the integrability of \mbox{$\log|z-w|\I$} with respect to $\mu\otimes\tilde\mu$.
To achieve equality in \eqref{e-log}, we need
\[ \int\e^{\frac{\i}{2}(kz+\bar k\bar z)}\d\mu(z) = \int\e^{\frac{\i}{2}(kz+\bar k\bar z)}\d\tilde\mu(z) \]
for all $k\in\C$, which reads $\mu=\tilde\mu$.
\end{proof}

\begin{definition}
A Borel probability measure $\mu_0$ on $D$ without point masses is called an equilibrium measure for the potential $V$ on $D\subset\C$ if 
\[ I(\mu_0)=\inf_{\mu\in\mathcal M(D)}I(\mu), \]
where $\mathcal M(D)$ denotes the set of all Borel probability measures on $D$ without point masses. We then set $I_0=I(\mu_0)$.
\end{definition}

Since $I(\frac1{\lambda_D(D)}\lambda_D)<\infty$, where $\lambda_D$ denotes the Lebesgue measure on $D$, $I_0$ is finite.

\begin{theorem}
Every continuous potential $V$ on a compact domain $D$ has a unique equilibrium measure.
\end{theorem}
\begin{proof}
To show the infimum is achieved, we choose a sequence $(\mu_n)_{n=1}^\infty$ in $\mathcal M(D)$ with $I(\mu_n)\to I_0$. 
Because of Lemma \ref{l:compactness}, there exists a convergent subsequence $(\mu_{n(k)})_{k=1}^\infty$ of $(\mu_n)_{n=1}^\infty$ and a Borel probability measure $\mu$ with $\mu_{n(k)}\to\mu$.

To prove that $I(\mu)=I_0$, we estimate with an arbitrary real constant $L$
\begin{align*}
\lim_{k\to\infty}I(\mu_{n(k)}) &= \lim_{k\to\infty}\int V(z)\d\mu_{n(k)}(z)+\lim_{k\to\infty}\iint\log|z-\zeta|\I\d\mu_{n(k)}(\zeta)\d\mu_{n(k)}(z) \\
&\ge\int V(z)\d\mu(z)+\lim_{k\to\infty}\iint\min\{\log|z-\zeta|\I,L\}\d\mu_{n(k)}(\zeta)\d\mu_{n(k)}(z).
\end{align*}
Approximating uniformly the second integrand according to the theorem of Stone-Weier\-stra\ss\ up to some $\varepsilon>0$ with a polynomial in $z$, $\bar z$, $\zeta$ and $\bar\zeta$ and using Fubini's theorem, we get
\[ I_0=\lim_{k\to\infty}I(\mu_{n(k)})\ge\int V(z)\d\mu(z)+\iint\min\{\log|z-\zeta|\I,L\}\d\mu(\zeta)\d\mu(z)-2\varepsilon. \]
Letting first $\varepsilon\to0$ and then $L\to\infty$ shows that $\mu$ has no point masses (otherwise the right hand side would diverge) and $I(\mu)=I_0$.

Next we want to show that there is exactly one measure $\mu\in\mathcal M(D)$ with $I(\mu)=I_0$. So suppose $\tilde\mu\in\mathcal M(D)$ fulfils $I(\tilde\mu)=I_0$, too. Then we consider the family 
\[ \mu_t = t\tilde\mu+(1-t)\mu = \mu+t(\tilde\mu-\mu),\quad t\in[0,1], \]
in $\mathcal M(D)$, expand with regard to Lemma \ref{l-integrability} the functional $I(\mu_t)$, and obtain
\begin{equation}
\begin{split}\label{e-I expanded}
I(\mu_t) &= I(\mu)+t\int\left(V(z)+2\int\log|z-\zeta|\I\d\mu(\zeta)\right)\d(\tilde\mu-\mu)(z) \\
&\qquad\qquad+t^2\iint\log|z-\zeta|\I\d(\tilde\mu-\mu)(\zeta)\d(\tilde\mu-\mu)(z). 
\end{split}
\end{equation}
Lemma \ref{l-integrability} now states that the coefficient of $t^2$ is non-negative, and so the function $t\mapsto I(\mu_t)$ is convex on $[0,1]$. In particular, for all $t\in[0,1]$,
\[ I(\mu_t) \le tI(\tilde\mu)+(1-t)I(\mu) = I_0, \]
which implies $I(\mu_t)=I_0$ for all $t\in[0,1]$. This requires the last summand in \eqref{e-I expanded} to vanish, and so, again with Lemma \ref{l-integrability}, we see that $\mu=\tilde\mu$.
\end{proof}

\begin{proposition}\label{p-variational principle}
The probability measure $\mu$ is the equilibrium measure for the potential $V$ on the domain $D$ if and only if the function
\begin{equation}\label{e-E}
E(z)=V(z)+2\int\log|z-\zeta|\I\d\mu(\zeta)
\end{equation}
fulfils the relation
\begin{equation}
\int E(z)\d\tilde\mu(z)\ge\int E(z)\d\mu(z)\quad\textrm{for all $\tilde\mu\in\mathcal M(D)$.} \label{e-variation 1}
\end{equation}
For the equilibrium measure $\mu$ and $E_0=\int E(z)\d\mu(z)$, we additionally have
\begin{equation}
E(z)=E_0\quad\textrm{$\mu$-almost everywhere.}\label{e-variation 2}
\end{equation}
\end{proposition}

\begin{proof}
Let us first assume $\mu$ is the equilibrium measure. Then, for an arbitrary measure $\tilde\mu\in\mathcal M(D)$, the condition
\[ \frac{\d}{\d t}I(\mu_t)\big|_{t=0} \ge 0,\quad\mu_t=t\tilde\mu+(1-t)\mu,\quad t\in[0,1], \]
has to hold. This means
\[ \int\left(V(z)+2\int\log|z-\zeta|\I\d\mu(\zeta)\right)\d(\tilde\mu-\mu)(z)\ge 0, \]
which immediately implies equation (\ref{e-variation 1}).

Assume on the other side that $\mu$ fulfils the condition (\ref{e-variation 1}). Then we obtain for the equilibrium measure $\mu_0$
\begin{align*}
I(\mu_0) &= I(\mu)+\int\left(V(z)+2\int\log|z-\zeta|\I\d\mu(\zeta)\right)\d(\mu_0-\mu)(z) \\
&\phantom{=\;I(\mu)}+\iint\log|z-\zeta|\I\d(\mu_0-\mu)(\zeta)\d(\mu_0-\mu)(z) \\
&\ge I(\mu),
\end{align*}
and so $\mu=\mu_0$.

To prove relation \eqref{e-variation 2}, we consider for the equilibrium measure $\mu$ the set
\[ B = \{z\in D\;|\;E(z) < E_0\}. \]
If $\mu(B)>0$, the variational principle (\ref{e-variation 1}) for the measure $\tilde\mu$ given by $\d\tilde\mu = \frac{\chi_B}{\mu(B)}\d\mu$ would yield
\[ E_0\le\int E(z)\frac{\chi_B(z)}{\mu(B)}\d\mu(z)<E_0, \]
and so $\mu(B)$ has to vanish. Thus, we get relation \eqref{e-variation 2}.
\end{proof}

\begin{corollary}\label{c-variational principle}
If for a measure $\mu\in\mathcal M(D)$ the function $E$ defined by equation \eqref{e-E} fulfils, for some real constant $E_0$, $E(z)=E_0$ on the support of $\mu$ and $E(z)\ge E_0$ everywhere, then $\mu$ is the equilibrium measure.
\end{corollary} 

We may give this variational principle an electrostatic meaning. We consider the equilibrium measure $\mu$ as density of charged particles in the plane with total charge $1$. Then, since $2\log|z|\I$ is the Green's function in two dimension, the function $E$ becomes the electrostatic potential produced by the external potential $V$ and the charged particles. Corollary \ref{c-variational principle} now tells us that the electrostatic potential is constant inside the charged area and increases outwards, so the configuration described by the measure $\mu$ is stable. The functional $I$ becomes in this setting the electrostatic potential energy of the configuration, which by construction is minimised by $\mu$. 

Out of this analogy, we see that for an absolutely continuous measure $\mu$ with $\d\mu(z)=\rho(z)\dd z$, the density $\rho$ will solve the Poisson equation $\triangle V=-4\pi\rho$ inside the support of $\mu$.

\subsection{The eigenvalue density}\label{s:eigenvalueDensity}
Let us now calculate the eigenvalue density and, more generally, the correlation functions $R_N^{(k)}$, $1\le k\le N$, out of the equilibrium measure $\mu$. To this end, we first consider the behaviour of the probability distribution $P_N$ around the equilibrium measure. So consider the sets
\[ A_{N,\eta}=\{\vec z\in D_0^N\;|\;I(\delta_{\vec z})\le I_0+\eta\},\quad\eta>0. \]
Let us further assume that the boundary $\partial D$ of the cut-off $D$ is a twice continuously differentiable curve.

\begin{lemma}\label{l:measureApproximation}
For all $\varepsilon>0$, there exists an absolutely continuous Borel probability measure $\mu_\varepsilon$ on $D$ so that $\d\mu_\varepsilon(z)=\rho_\varepsilon(z)\dd z$ with $\rho_\varepsilon\in C(D)$ and
\[ I(\mu_\varepsilon)\le I_0+\varepsilon. \]
\end{lemma}
\begin{proof}
Set $B_\varepsilon^D(z)=B_\varepsilon(z)\cap D$. Since the boundary of $D$ is chosen as twice continuously differentiable curve, we find an $\varepsilon_0>0$ such that the circle $\partial B_\varepsilon(z)$ cuts the boundary $\partial D$ of $D$ for all $z\in D$ and $\varepsilon\in(0,\varepsilon_0)$ at most twice. Then, for any $\delta\in(0,\frac12)$, we find an $\varepsilon_1\in(0,\varepsilon_0)$ such that the area $\kappa_\varepsilon(z)=\lambda(B_\varepsilon^D(z))$ of $B_\varepsilon^D(z)$ fulfils
\[ \delta\le\frac{\kappa_\varepsilon(z)}{\pi\varepsilon^2}\le1\quad\textrm{for all $z\in D$ and $\varepsilon\in(0,\varepsilon_1)$}. \]

Let now $\mu$ be the equilibrium measure and set $\tilde\rho_\varepsilon(z)=\frac1{\kappa_\varepsilon(z)}\mu(B_\varepsilon^D(z))$. Then, for all continuous functions $\varphi\in C(D)$,
\begin{align*}
\left|\int_D\varphi(z)\tilde\rho_\varepsilon(z)\dd z-\int_D\varphi(\zeta)\d\mu(\zeta)\right|&\le\int_D\left|\int_{B_\varepsilon^D(\zeta)}\frac{\varphi(z)}{\kappa_\varepsilon(z)}\dd z-\varphi(\zeta)\right|\d\mu(\zeta) \\
&\le\int_D\max_{z\in B_\varepsilon^D(\zeta)}\left|\frac{\kappa_\varepsilon(\zeta)}{\kappa_\varepsilon(z)}\varphi(z)-\varphi(\zeta)\right|\d\mu(\zeta)\to0\quad(\varepsilon\to0).
\end{align*}
So the measure $\tilde\mu_\varepsilon$, defined by $\d\tilde\mu_\varepsilon(z)=\tilde\rho_\varepsilon(z)\dd z$, converges weakly to the equilibrium measure $\mu$. Moreover, we have
\begin{align*}
&\left|\iint_{D^2}\log|z-w|\I\tilde\rho_\varepsilon(z)\tilde\rho_\varepsilon(w)\dd z\dd w-\iint\log|\zeta-\omega|\I\d\mu(\zeta)\d\mu(\omega)\right| \\
&\qquad\qquad\le\iint\left|\iint_{B_\varepsilon^D(\zeta)\times B_\varepsilon^D(\omega)}\frac1{\kappa_\varepsilon(z)\kappa_\varepsilon(w)}\log\left|\frac{z-w}{\zeta-\omega}\right|\I\dd z\dd w\right|\d\mu(\zeta)\d\mu(\omega) \\
&\qquad\qquad\qquad+\iint\left|\iint_{B_\varepsilon^D(\zeta)\times B_\varepsilon^D(\omega)}\frac1{\kappa_\varepsilon(z)\kappa_\varepsilon(w)}\dd z\dd w-1\right|\left|\log|\zeta-\omega|\right|\d\mu(\zeta)\d\mu(\omega).
\end{align*}
Now the second term tends to zero in the limit $\varepsilon\to0$, and because of
\begin{align*}
&\left|\iint_{B_\varepsilon^D(\zeta)\times B_\varepsilon^D(\omega)}\frac1{\kappa_\varepsilon(z)\kappa_\varepsilon(w)}\log\left|\frac{z-w}{\zeta-\omega}\right|\I\dd z\dd w\right| \\
&\qquad\qquad\le\frac{|\zeta-\omega|^4}{\delta^2\pi^2\varepsilon^4}\iint_{B_{\varepsilon/|\zeta-\omega|}(0)^2}\left|\log\left|1+z-w\right|\right|\dd z\dd w \\
&\qquad\qquad\le\frac{|\zeta-\omega|^2}{\delta^2\pi\varepsilon^2}\int_{B_{2\varepsilon/|\zeta-\omega|}(0)}\left|\log\left|1+v\right|\right|\dd v 
\le C\left|\log\left|1+\frac{2\varepsilon}{|\zeta-\omega|}\right|\right|\to0\quad(\varepsilon\to0),
\end{align*}
for some constant $C>0$, the first term, too, vanishes in the limit $\varepsilon\to0$. Therefore, $I(\tilde\mu_\varepsilon)\to I_0$ for $\varepsilon\to0$. Hence, we find for any $\varepsilon>0$ some $\tilde\varepsilon>0$ such that the absolutely continuous measure $\mu_\varepsilon=\frac1{\tilde\mu_{\tilde\varepsilon}(D)}\tilde\mu_{\tilde\varepsilon}$ fulfils $I(\mu_\varepsilon)\le I_0+\varepsilon$.
\end{proof}
 
\begin{lemma}\label{l:probability peak}
The probability $P_N(D^N\setminus A_{N,\eta})$ drops for $N\to\infty$ to zero faster than $\e^{-\frac12N^2\eta}$.
\end{lemma}
\begin{proof}
Using Lemma \ref{l:measureApproximation}, we choose a Borel probability measure $\mu_\eta$ with 
\[ I(\mu_\eta)\le I_0+\frac\eta4\quad\textrm{and}\quad\d\mu_\eta(z)=\rho_\eta(z)\dd z,\quad \rho_\eta\in C(D). \]
Then, with Jensen's theorem and $\int I(\delta_{\vec z})\prod\d\mu_\eta(z_i)\le I_0+\frac\eta4+\ord(1)$ for $N\to\infty$, we get
\[ Z_N\ge\int_{\{\vec z\in D_0^N\,|\,\rho_\eta(z_i)\ne0\;\forall i\}}\hspace{-6em}\e^{-N^2I(\delta_{\vec z})-\sum_{i=1}^N\log\rho_\eta(z_i)}\prod_{i=1}^N\d\mu_\eta(z_i)\ge\e^{-N^2(I_0+\frac\eta4)+\ord(N^2)}, \]
and therefore,
\[ P_N(D^N\setminus A_{N,\eta})\le\int_{D_0^N}\e^{N^2(I_0+\frac\eta4)+\ord(N^2)-N^2(I_0+\eta)}\prod_{i=1}^N\dd{z_i}=\ord(\e^{-\frac12N^2\eta}). \qedhere \]
\end{proof}

\begin{lemma}\label{l:convergent measure}
For $\eta>0$, any sequence $(\vec z_N)_{N\in\N}$, $\vec z_N\in A_{N,\eta}$, and any convergent subsequence $(\delta_{\vec z_{N(n)}})_{n\in\N}$ of the corresponding $\delta$-measures, we have 
\[ I_0\le I(\nu_\eta)\le I_0+\eta,\quad\textrm{where}\quad\nu_\eta=\lim_{n\to\infty}\delta_{\vec z_{N(n)}}\in\mathcal M(D). \]
\end{lemma}
\begin{proof} 
Using $\vec z_{N(n)}\in A_{N(n),\eta}$, we obtain with the cut-off $L\in\R$:
\begin{align*}
I_0+\eta &\ge \frac{1}{N(n)}\sum_{i=1}^{N(n)}V(z_{N(n),i})+\frac1{N(n)^2}\!\!\!\!\sum_{\phantom{(n)}1\le i\ne j\le N(n)}\!\!\!\!\min\{\log|z_{N(n),i}-z_{N(n),j}|\I,L\} \\
&=\int V(z)\d\delta_{\vec z_{N(n)}}(z)+\iint\min\{\log|z-\zeta|\I,L\}\d\delta_{\vec z_{N(n)}}(\zeta)\d\delta_{\vec z_{N(n)}}(z)-\frac{L}{N(n)}.
\end{align*}
Sending first $n$ and then $L$ to infinity brings us to $\nu_\eta\in\mathcal M(D)$, and therefore $I_0\le I(\nu_\eta)\le I_0+\eta$.
\end{proof}

\begin{theorem}\label{t-eigenvalue distribution}
For all $\varphi\in C(D^k)$, we have the equality
\begin{equation}\label{e-eigenvalue distribution}
\lim_{N\to\infty}\int_{D^k}\frac1{N^k}\varphi\big((z_i)_{i=1}^k\big)R^{(k)}_N\big((z_i)_{i=1}^k\big)\prod_{i=1}^k\dd{z_i}=\int\varphi\big((z_i)_{i=1}^k\big)\prod_{i=1}^k\d\mu(z_i).
\end{equation}
That is the measure $\frac1{N^k}R_N^{(k)}\big((z_i)_{i=1}^k\big)\prod_{i=1}^k\dd{z_i}$ on $D^k$ converges weakly to \mbox{$\prod_{i=1}^k\d\mu(z_i)$.}
\end{theorem}
\begin{proof}
Substituting in the left-hand-side of equation \eqref{e-eigenvalue distribution} the definition of the correlation functions and turning our attention on the highest order in $N$, we obtain (because $P_N$ is invariant under the symmetric group)
\begin{align*}
\left<\varphi,\frac1{N^k}R_N^{(k)}\right>&=\frac1{N^k}\int_{D^k}\varphi(\vec z)R^{(k)}_N(\vec z)\ddd{k}{\vec z} \\
&= \frac{1}{N^k}\sum_{i_1,\ldots,i_k=1}^N\int_{D^N}\varphi\big((z_{i_j})_{j=1}^k\big)P_N(\vec z)\ddd{N}{\vec z}+\ord(1).
\end{align*}

Because of Lemma \ref{l:probability peak}, we may now integrate over the set $A_{N,\eta}$ instead of $D^N$. So let the continuous function $\frac1{N^k}\sum\varphi\big((z_{i_j})_{j=1}^k\big)$ take its maximum on the compact set $A_{N,\eta}$ at $\vec\zeta$, and set $\nu_{N,\eta} = \delta_{\vec\zeta}$. Then
\[ \left<\varphi,\frac{1}{N^k}R_N^{(k)}\right>\le\frac1{N^k}\sum_{i_1,\ldots,i_k=1}^N\varphi\big((\zeta_{i_j})_{j=1}^k\big)+\ord(1)=\int\varphi\big((z_i)_{i=1}^k\big)\prod_{i=1}^k\d\nu_{N,\eta}(z_i)+\ord(1). \]
Because of Lemma \ref{l:compactness}, we find a convergent subsequence $\nu_{N(n),\eta}\to\nu_\eta$ $(n\to\infty)$ with
\[ \varlimsup_{N\to\infty}\left<\varphi,\frac{1}{N^k}R_N^{(k)}\right>\le\int\varphi\big((z_i)_{i=1}^k\big)\prod_{i=1}^k\d\nu_\eta(z_i). \]

Now, because of Lemma \ref{l:convergent measure}, letting $\eta\to0$, a subsequence of $\nu_\eta$ converges to the equilibrium measure $\mu$, and thus,
\[ \varlimsup_{N\to\infty}\left<\varphi,\frac{1}{N^k}R_N^{(k)}\right>\le\int\varphi\big((z_i)_{i=1}^k\big)\prod_{i=1}^k\d\mu(z_i). \]
Arguing on the same way for the limes inferior concludes the proof.
\end{proof}



\pagebreak

\section{The equilibrium measure for a polynomial curve}
\subsection{The general result}
\begin{theorem}\label{t:equilibrium measure of a polynomial curve}
Let $d\in\N$. Then, for any set $(t_k)_{k=1}^\infty\subset\C$ with $t_1=0$, $|t_2|<\frac12$, and $t_k=0$ for $k>d+1$ and any compact domain $D\subset\C$ containing the origin as interior point such that the function
\[ U(z) = |z|^2-2\RE\sum_{k=2}^{d+1}t_kz^k \]
is positive on $D\setminus\{0\}$, there exists a $\delta>0$ so that for all $0<t_0<\delta$, we have $(t_k)_{k=0}^\infty\in\mathcal T_d$, and the equilibrium measure $\mu$ for the potential $V=\frac1{t_0}U$ on $D$ is given by
\[ \mu=\frac1{\pi t_0}\lambda_{D_+}, \]
where $D_+$ denotes the interior domain of the polynomial curve $\gamma$ defined by the harmonic moments $(t_k)_{k=0}^\infty$, and $\lambda_{D_+}$ is the Lebesgue measure on $D_+$.
\end{theorem}
Therefore, the normal matrix model solves the problem of finding the interior harmonic moments $(v_k)_{k=1}^\infty$ out of the exterior harmonic moments $(t_k)_{k=0}^\infty$, namely (as can be seen from Theorem \ref{t:equilibrium measure of a polynomial curve} and \ref{t-eigenvalue distribution})
\[ v_k=\lim_{N\to\infty}\frac{t_0}N\int_{\mathcal N_N(D)}\tr(M^k)\mathcal P_N(M)\d M. \]


\begin{lemma}\label{l-gradient}
Let 
\[ V(z)=\frac1{t_0}\left(|z|^2-2\RE\sum_{k=1}^{d+1}t_kz^k\right),\quad (t_k)_{k=0}^\infty\in\mathcal T_d, \]
and denote by $\gamma$ the polynomial curve defined by the harmonic moments $(t_k)_{k=0}^\infty$. Then, in the interior domain $D_+$ of $\gamma$, the function 
\[ E(z)=V(z)+\frac2{\pi t_0}\int_{D_+}\log\left|\frac z\zeta-1\right|\I\dd\zeta \]
is equal to zero, and in the exterior domain $\C\setminus D_+$, its gradient reads
\begin{equation}\label{e-gradient}
\partial_{\bar z}E(z)=\frac1{t_0}(z-\rho(z)),
\end{equation}
where $\rho$ is the Schwarz reflection on $\gamma$.
\end{lemma}
\begin{proof}
To verify the first statement, we use Green's theorem and obtain
\[ \frac2\pi\int_{D_+}\log\left|\frac z\zeta-1\right|\I\dd\zeta = -|z|^2+\RE\frac{1}{2\pi\i}\oint_\gamma\left(\log\left|\frac z\zeta-1\right|\I\bar\zeta+\frac{|\zeta|^2}{\zeta-z}\right)\d\zeta. \]
Integration by parts of the second integrand yields immediately
\[ \frac2\pi\int_{D_+}\log\left|\frac z\zeta-1\right|\I\dd\zeta = -|z|^2-2\RE\frac{1}{2\pi\i}\oint_\gamma\log\left(1-\frac z\zeta\right)\bar\zeta\d\zeta, \]
and expanding the logarithm around $z=0$ leads us to $E(z)=0$ for all $z\in D_+$.

For the proof of the second part, we write the Schwarz function $S$ of the curve $\gamma$ as $S=S_{\mathrm i}+S_{\mathrm e}$, where $S_{\mathrm i}$ is analytic in $D_+$ and $S_{\mathrm e}$ in the complement $\C\setminus D_+$.
From Proposition \ref{p:expansion of schwarz function}, we already know that $S_{\mathrm i}(z)=\sum_{k=1}^{d+1}kt_kz^{k-1}$. 
For the function $S_{\mathrm e}$, we find with Cauchy's integral and Stokes' formula
\[ S_{\mathrm e}(z) = -\frac1{2\pi\i}\oint_\gamma\frac{\bar\zeta-S_{\mathrm i}(\zeta)}{\zeta-z}\d\zeta  = \frac1\pi\int_{D_+}\frac1{z-\zeta}\dd\zeta,\quad z\in\C\setminus D_+. \]


Therefore, for all $z\in \C\setminus D_+$,
\[ \partial_{\bar z}E(z) = \frac1{t_0}\left(z-\overline{S_{\mathrm i}(z)}-\overline{S_{\mathrm e}(z)}\right) = \frac1{t_0}\left(z-\rho(z)\right).\qedhere \]
\end{proof}


\subsection{The Gaussian case}
In this case, where the polynomial curve is an ellipse, we are able to calculate the equilibrium measure on $\C$ explicitly. 
\begin{proposition}\label{p:equilibriumMeasureEllipse}
For any $t_0>0$ and any $t_2\in\C$ with $|t_2|<\frac12$, the equilibrium measure $\mu$ for the potential
\[ V(z)=\frac{1}{t_0}(|z|^2-t_2z^2-\bar t_2\bar z^2) \]
on $\C$ is given by
\[ \mu = \frac{1}{b_1b_2}\lambda_{D_+}, \]
where $D_+$ denotes the interior of the ellipse
\begin{equation}\label{e-ellipse}
\frac{\RE(\sqrt{\bar t_2}z)^2}{b_1^2}+\frac{\IM(\sqrt{\bar t_2}z)^2}{b_2^2}=|t_2|,\quad b_1=\sqrt{\frac{1+2|t_2|}{1-2|t_2|}t_0},\quad b_2=\sqrt{\frac{1-2|t_2|}{1+2|t_2|}t_0}.
\end{equation}
\end{proposition}
\begin{proof}
As polynomial curve, the ellipse (\ref{e-ellipse}) has the parametrisation
\begin{equation}\label{e-param. ell.}
h(w) = r(w+2t_2w\I),\quad r=\frac{b_1+b_2}2.
\end{equation}
We check that the given measure $\mu$ is the equilibrium measure by verifying the conditions of Corollary \ref{c-variational principle}. To this end, let us introduce for $|w|>1$ the function
\[ \mathcal E(w) = E(h(w)),\quad E(z) = V(z)+\frac2{\pi t_0}\int_{D_+}\log\left|\frac z{\zeta}-1\right|\I\dd\zeta. \]
Integrating equation (\ref{e-gradient}) and its complex conjugate analogon, we get for $\mathcal E(w)$ the expression 
\[ \mathcal E(w)=\frac1{t_0}\left(|h(w)|^2-|h(1)|^2-2\RE\int_1^w\bar h(\omega\I)h'(\omega)\d\omega\right). \]
Substituting in this expression relation (\ref{e-param. ell.}) for $h$, we obtain
\begin{align*}
t_0\mathcal E(w) &= r^2\big(|w+2t_2w\I|^2-|1+2t_2|^2\big) \\
&\qquad\qquad-2r^2\RE(\bar t_2w^2+(1-4|t_2|^2)\log w+t_2w^{-2}-(t_2+\bar t_2)) \\
&= r^2\big(|w|^2-1-4|t_2|^2+4|t_2|^2|w|^{-2}+(1-4|t_2|^2)\log(|w|^{-2})\big) \\
&\qquad\qquad+2r^2\RE(2t_2\bar ww\I-t_2(\bar ww\I)|w|^{-2}-t_2(\bar ww\I)|w|^2) \\
&=r^2(|w|^2-1)(1-4|t_2|^2|w|^{-2})+r^2(1-4|t_2|^2)\log(|w|^{-2}) \\
&\qquad\qquad-2r^2(|w|^2-1)(1-|w|^{-2})\RE(t_2\bar ww\I) \\
&=r^2(|w|^2-1)(1-2|t_2|)(1+2|t_2||w|^{-2})+r^2(1-4|t_2|^2)\log(|w|^{-2}) \\
&\qquad\qquad+2r^2(|w|^2-1)(1-|w|^{-2})(|t_2|-\RE(t_2\bar ww\I)).
\end{align*}
We are now ready to start estimating $\mathcal E(w)$ for $|w|>1$ and $|t_2|<\frac12$. The last bracket we can estimate by
\[ |t_2|-\RE(t_2\bar ww\I)\ge |t_2|-|t_2\bar ww\I|=0. \]
There therefore remains to show
\[ (|w|^2-1)(1+2|t_2||w|^{-2})+(1+2|t_2|^2)\log(|w|^{-2})\ge 0, \]
which follows immediately out of the following lemma.
\begin{lemma}
For $\alpha\in[0,1]$, the function 
\[ f(x) = (x-1)(1+\alpha x\I)-(1+\alpha)\log x \]
is non-negative on the interval $[1,\infty)$.
\end{lemma}
\begin{proof}
For the function $f$ and its derivatives $f'$ and $f''$, we have
\[ f(1)=0,\quad f'(1)=0,\quad\textrm{and}\quad f''(x)=\frac{1}{x^2}\left(1+\alpha-\frac{2\alpha}{x}\right)\ge0.\qedhere \]
\end{proof}

So we showed that $E(z)\ge0$ for all $z\in\C\setminus D_+$. Because of Lemma \ref{l-gradient}, we also know that $E$ is zero in the interior domain $D_+$, and we therefore can apply Corollary \ref{c-variational principle} to see that $\mu$ is indeed the equilibrium measure for the potential $V$ on $\C$. 
\end{proof}

\subsection{The proof of Theorem \ref{t:equilibrium measure of a polynomial curve}} 
As in the Gaussian case, we are going to show that the measure $\mu$ given in the theorem fulfils the conditions of Corollary \ref{c-variational principle} and is therefore the uniquely defined equilibrium measure.

Because Theorem \ref{t:equilibrium measure of a polynomial curve} is only valid for interior domains with small area, we are going to consider the asymptotical behaviour $t_0\to0$, where the harmonic moments $(t_k)_{k=1}^\infty$ are kept fixed. 
To catch the asymptotical behaviour of the corresponding polynomial curve $\gamma$, let us write its parametrisation $h$ as in the proof of Theorem \ref{t:uniquePolCurve}:
\[ h(w) = rw+\sum_{j=0}^dr^j\alpha_jw^{-j}. \]
Then, as we know from the proof of Theorem \ref{t:uniquePolCurve}, we have for $r\to0$
\[ t_0=(1-4|t_2|^2)r^2+\Ord(r^4)\quad\textrm{and}\quad\alpha_j=(j+1)\bar t_{j+1}+\Ord(r^2),\quad 0\le j\le d. \]

\begin{lemma}
The critical radius $R$ of $\gamma$ is asymptotically constant for $r\to 0$:
\[ R=\sqrt{|\alpha_1|}+\Ord(r). \]
\end{lemma}
\begin{proof}
The roots of the function $h'(w)=r-\sum_{j=1}^djr^j\alpha_jw^{-j-1}$ are in zeroth order at $\pm\sqrt{\alpha_1}$ and ($(n-1)$-times degenerated) at zero.
\end{proof}

We consider now for $z\in D$ and $w\in h\I(D\setminus D_+)$ the functions
\begin{align}
E(z) &= V(z)+\frac2{\pi t_0}\int_{D_+}\log\left|\frac z\zeta-1\right|\I\dd\zeta\quad\textrm{and} \label{e:energy} \\
\mathcal E(w) &= E(h(w)) = \frac1{t_0}\left(|h(w)|^2-|h(1)|^2-2\RE\int_1^w\bar h(\omega\I)h'(\omega)\d\omega\right).
\end{align}

We already showed in Lemma \ref{l-gradient} that $E(z)=0$ for all $z\in D_+$, so the first condition of Corollary \ref{c-variational principle} is satisfied (this is essentially the way we have chosen our potential $V$).

Now, again with Lemma \ref{l-gradient}, we see that $E\ge0$ in the vicinity of the curve $\gamma$, strictly speaking in the domain $h(B_{R\I}\setminus\bar B_1)$, where $R$ denotes the critical radius of $\gamma$. Indeed, if we look at the connected components of the contour lines of the function $E$ (which are smooth curves in the considered domain because there $\partial_{\bar z}E\ne 0$), we see that the gradient vector $\partial_{\bar z}E$ always points outwards, that is into the exterior domain of the contour line. Therefore, the value of $E$ on the contour lines is increasing outwards as desired.

A bit farther from the curve, that is for $R\I\le|w|<r^{-\varepsilon}$ for some $\varepsilon\in(0,\frac13)$, the function $\mathcal E$ equals asymptotically the one of the corresponding ellipse $h\0(w) = rw+\alpha_0+r\alpha_1w\I$, we denote it by $\mathcal E\0$. Indeed, remarking that we have $t_0\0 = t_0+\Ord(r^4)$ for the area $\pi t_0\0$ of this ellipse and that $\mathcal E\0(w)=\Ord(r^{-2\varepsilon})$, we uniformly in $w$ obtain
\begin{align*}
\mathcal E(w)-\mathcal E\0(w) &= \frac1{t_0}(|h(w)|^2-|h\0(w)|^2-|h(1)|^2+|h\0(1)|^2) \\
&\qquad+\frac2{t_0}\RE\int_1^w(\bar h(\omega\I)-\bar h\0(\omega\I))h'(\omega)\d\omega \\
&\qquad+\frac2{t_0}\RE\int_1^w(h'(\omega)-{h\0}'(\omega))\bar h\0(\omega\I)\d\omega \\
&\qquad+\frac1{t_0}\left(t_0\0-t_0\right)\mathcal E\0(w) \\
&= \Ord(r)+\Ord(r^{1-3\varepsilon})+\Ord(r)+\Ord(r^{2-2\varepsilon}) \\
&\to0\quad(r\to0).
\end{align*}
Now, from the proof of Proposition \ref{p:equilibriumMeasureEllipse}, we know that there exists a $C>0$ such that $\mathcal E\0(w)\ge C$ for all $w\in\C\setminus B_{R\I}$ and any $r>0$. Therefore, we may choose $r$ so small that $|\mathcal E(w)-\mathcal E\0(w)|<\mathcal E\0(w)$ and so $\mathcal E(w)>0$ for all $w\in B_{r^{-\varepsilon}}\setminus B_{R\I}$.

There remains the domain where $|w|\ge r^{-\varepsilon}$. 
For $k\ge2$, we obtain
\begin{align*}
h(w)^k-(rw)^k &= \sum_{\ell=1}^k\begin{pmatrix}k\\\ell\end{pmatrix}(rw)^{k-\ell}\left(\sum_{j=0}^dr^j\alpha_jw^{-j}\right)^\ell \\
&=\sum_{\ell=1}^k\begin{pmatrix}k\\\ell\end{pmatrix}r^kw^{k-2\ell}\left(\sum_{j=0}^dr^{j-1}\alpha_jw^{1-j}\right)^\ell 
=\Ord(r^2).
\end{align*}
Since 
\begin{align*}
U(z)&=|z|^2-t_2z^2-\bar t_2\bar z^2+\ord(|z|^2) \\
&=\frac12\left(|z-2\bar t_2\bar z|^2+(1-4|t_2|^2)|z|^2\right)+\ord(|z|^2)
\end{align*}
for $z\to 0$ and $U>0$ on $D\setminus\{0\}$, there exists a constant $c>0$ such that $U(z)\ge c|z|^2$ for all $z\in D$.

Therefore, for $r\to0$,
\[ V(h(w))=V(rw)+\Ord(1)\ge\frac c{t_0}|rw|^2+\Ord(1), \]
which tends to infinity at least as $r^{-2\varepsilon}$.

On the other hand, the integral over the logarithm in \eqref{e:energy} diverges for $r\to0$ only as $\log r$. So, for $r$ small enough, we have $E(h(w))>0$ for all $w\in h\I(D\setminus D_+)\setminus B_{r^{-\varepsilon}}$.

This proves $E(z)\ge0$ for all $z\in D$ and therefore, with Corollary \ref{c-variational principle}, that $\mu$ is the equilibrium measure.

\subsection{Shifting the origin}
As a little generalisation, we consider the case where $t_1\ne0$, which corresponds to a shift of the origin. So we need to define the harmonic moments also for curves which do not encircle the origin.

\begin{definition}
Let
\[ h(w)=rw+\sum_{j=0}^da_jw^{-j} \]
be the parametrisation of a polynomial curve of degree $d$. Then the exterior harmonic moments $(t_k)_{k=1}^{d+1}$ are given by the equation system \eqref{e:harmonicMoments}. All other exterior harmonic moments are set to zero.
\end{definition}
Proposition \ref{p:polCurve/harmMom} tells us that this definition coincides with the previous one if the origin is in the interior domain of the curve.
\begin{corollary}
Let $d\in\N$. Then, for any set $(t_k)_{k=1}^\infty\subset\C$ with $|t_2|<\frac12$ and $t_k=0$ for $k>d+1$ and any compact domain $D\subset\C$ such that the function
\[ U(z)=|z|^2-2\RE\sum_{k=1}^{d+1}t_kz^k \]
has exactly one absolute minimum in the interior of $D$, there exists a $\delta>0$ so that for all $0<t_0<\delta$ the equilibrium measure $\mu$ for $V=\frac1{t_0}U$ on $D$ is given by
\[ \mu=\frac1{\pi t_0}\lambda_{D_+}, \]
where $D_+$ denotes the interior domain of the polynomial curve $\gamma$ defined by the harmonic moments $(t_k)_{k=0}^\infty$.
\end{corollary}
\begin{proof}
Let us first shift the origin by $a_0$, such that $V(z+a_0)$ has its absolute minimum at $z=0$. Thereby, the potential gets the form
\[ V(z+a_0)=\frac1{t_0}\left(|z|^2-2\RE\sum_{k=2}^{d+1}t'_kz^k\right)+V(a_0), \]
where the $t'_k$ are the harmonic moments of the shifted curve $\gamma-a_0$. Indeed, the coefficients of $V$ and the harmonic moments depend polynomially on the shift $a_0$. Because we know them to coincide as long as the origin is inside $D_+$, they do so for all $a_0$. 

Applying now Theorem \ref{t:equilibrium measure of a polynomial curve} for the shifted potential gives the desired result.
\end{proof}


\section{The orthogonal polynomials corresponding to polynomial curves}
\subsection{The zeros of the orthogonal polynomials}\label{s:zerosOgPol}
Let us consider a potential $V$ on some compact domain $D$ such that Theorem \ref{t:equilibrium measure of a polynomial curve} is applicable.
As in Section \ref{s:eigenvalueDensity}, we will assume that the boundary $\partial D$ of the cut-off $D$ is a twice continuously differentiable curve.

\begin{lemma}\label{l:asymptoticOgPol}
For $x\in(0,1]$, the orthogonal polynomials $(p_{n,N})_{n=0}^\infty$ for the potential $V$ fulfil for all $z\notin D$
\[ \lim_{N\to\infty, \frac nN\to x}\frac1n\log|p_{n,N}(z)|=\int\log|z-\zeta|\d\mu_x(\zeta), \]
where $\mu_x$ is the equilibrium measure for the potential $\frac1xV$. 

Here and in the following, the limit $N\to\infty$, $\frac nN\to x$ shall be understood as $N,n\to\infty$ such that $|\frac nN-x|=\ord(\frac1N)$.
\end{lemma}
\begin{proof}
Using expression \eqref{e:ogPolInt} for $p_{n,N}$, we get for any $z\notin D$
\begin{align*}
&\frac1{z^n}\,p_{n,N}(z)\,\e^{-n\int\log(1-\frac\zeta z)\d\mu_x(\zeta)} \\
&\qquad\qquad=\int_{D^n}\e^{\sum_{i=1}^n\log(1-\frac{w_i}z)-n\int\log(1-\frac\zeta z)\d\mu_x(\zeta)-(N-\frac nx)\sum_{i=1}^nV(w_i)}P_{n,\frac1xV}(\vec w)\ddd n{\vec w}.
\end{align*}
Let now $\vec w_n$ denote the position of the maximum of the real part of the function 
\[ \vec w\mapsto\e^{\sum_{i=1}^n\log(1-\frac{w_i}z)-n\int\log(1-\frac\zeta z)\d\mu_x(\zeta)-(N-\frac nx)\sum_{i=1}^nV(w_i)} \]
on the compact set 
\[ \{\vec w\in D_0^n\;|\;I_{\frac1xV}(\delta_{\vec w})\le I_{\frac1xV}(\mu_x)+n^{-\frac34}\}. \]
Then, as in the proof of Theorem \ref{t-eigenvalue distribution},
\begin{align*}
&\RE\left(\frac1{z^n}\,p_{n,N}(z)\,\e^{-n\int\log(1-\frac\zeta z)\d\mu_x(\zeta)}\right) \\
&\qquad\qquad\le\RE\left(\e^{n\int\log(1-\frac\zeta z)\d\delta_{\vec w_n}(\zeta)-n\int\log(1-\frac\zeta z)\d\mu_x(\zeta)+\ord(n)}+\ord(\e^{-\frac14n^\frac54})\right).
\end{align*}
Because of Lemma \ref{l:convergent measure}, we get in the limit $N\to\infty$, $\frac nN\to x$ that 
\[ \RE\left(\frac1{z^n}\,p_{n,N}(z)\,\e^{-n\int\log(1-\frac\zeta z)\d\mu_x(\zeta)}\right) \le \e^{\ord(n)}+\ord(\e^{-\frac14n^{\frac54}}). \]

Doing the analogue estimates also from below and for the imaginary part, we get the predicted equation.
\end{proof}

\begin{lemma}\label{l:zerosOgPol}
In the limit $N\to\infty$ and $\frac nN\to x\in(0,1]$, almost all zeros of the orthogonal polynomials $p_{n,N}$ of $V$, $n\in\N_0$, lie inside the cut-off $D$.
\end{lemma}
\begin{proof}
Since the function 
\[ \frac1n\log|p_{n,N}(z)| = \int\log|z-\zeta|\d\delta_{\vec\xi_{n,N}}, \]
$\vec\xi_{n,N}$ denoting the zeros of $p_{n,N}$, converges outside the domain $D$ pointwise to the harmonic function $\int\log|z-\zeta|\d\mu_x(\zeta)$, where $\mu_x$ is the equilibrium measure for the potential $\frac1xV$, we find that in each compact area $K$ outside of $D$ the ratio between zeros of $p_{n,N}$ lying in $K$ and $N$ tends to zero. So the percentage of zeros of $p_{n,N}$ inside the domain $D$ is equal to 
\[ \lim_{N\to\infty,\frac nN\to x}\frac1{2\pi\i n}\oint_\alpha\frac{\partial_zp_{n,N}(z)}{p_{n,N}(z)}\d z=\frac1{2\pi\i xt_0}\oint_\alpha S_{x,\mathrm e}(z)\d z = 1, \]
where $\alpha$ is some curve encircling $D$ and omitting all zeros of the orthogonal polynomials and, as in the proof of Lemma \ref{l-gradient}, $S_{x,\mathrm e}(z)=S_x(z)-\sum_{k=1}^{d+1}kt_kz^{k-1}$, where $S_x$ denotes the Schwarz function of the boundary curve of the support of the equilibrium measure $\mu_x$.
\end{proof}

\begin{theorem}
Let $\vec\xi_{n,N}$ denote the zeros of the $n$-th orthogonal polynomial $p_{n,N}$ of $V$. Let further $\sigma_x$ be the limit of a convergent subsequence of the $\delta$-measures $\delta_{\vec\xi_{n,N}}$ in the limit $N\to\infty$, $\frac nN\to x\in(0,1]$. If $\sigma_x$ is supported on a tree-like graph $\mathcal G$, whose edges are smooth one-dimensional manifolds, then 
\begin{enumerate}
\item
the edges lie on branch cuts of the Schwarz function $S_x$ of the curve $\gamma_x$ defined by the exterior harmonic moments $(t_k)_{k=1}^\infty$ and the encircled area $\pi xt_0$, which are chosen such that the discontinuity $\delta S_x$ of $S_x$ fulfils
\begin{equation}\label{e:branchCuts}
\RE(\delta S_x(\alpha(t))\dot\alpha(t))\equiv0,
\end{equation}
for any (local) parametrisation $\alpha$ of the graph $\mathcal G$,
\item
$\sigma_x$ has, unless $V(z)=\frac1{t_0}|z|^2$ where $\sigma_x=\delta_0$, no point masses, and 
\item
along the edges we have, with $\d z$ being the infinitesimal line element on the edges,
\begin{equation}\label{e:densityZeros}
\d\sigma_x(z) = \rho_x(z)\d z,\quad\textrm{where}\quad\rho_x(z)=\frac{\delta S_x(z)}{2\pi\i xt_0}.
\end{equation}
\end{enumerate}
\end{theorem}
\begin{proof}
Lemma \ref{l:asymptoticOgPol} tells us that the limiting measure $\sigma_x$ fulfils for all $z\in\C\setminus D$ the equation
\[ \int\log|z-\zeta|\d\sigma_x(\zeta)=\int\log|z-\zeta|\d\mu_x(\zeta). \]
Then, differentiating and analytic continuing this identity, gives us
\begin{equation}\label{e:densityZerosSchwarz}
\int\frac1{z-\zeta}\d\sigma_x(\zeta)=\frac1{xt_0}S_{x,\mathrm e}(z),
\end{equation}
for all $z\in\C\setminus\mathcal G$, where, as in the proof of Lemma \ref{l-gradient}, $S_{x,\mathrm e}(z)=S_x(z)-\sum_{k=1}^{d+1}kt_kz^{k-1}$.

If not all exterior harmonic moments are zero, the Schwarz function $S_x$ is according to Proposition \ref{p:singularitiesSchwarz} finite everywhere, and therefore the measure $\sigma_x$ must not have any point masses. Setting therefore $\d\sigma_x(z)=\rho_x(z)\d z$ and comparing the discontinuities in equation \eqref{e:densityZerosSchwarz} along the edges of the graph $\mathcal G$, we find the relation \eqref{e:densityZeros}.

Since $\int_\alpha\d\sigma_x$ has to be a real number between zero and one, the edges of the graph for all segments $\alpha$ of the graph $\mathcal G$, and therefore our choice of branch cuts for the Schwarz function, have to fulfil the equation \eqref{e:branchCuts}.
\end{proof}

In view of the examples given in Section \ref{s:application}, it seems plausible to us to expect this behaviour for the orthogonal polynomials in general.
\begin{conjecture}
The sequence $(\delta_{\vec\xi_{n,N}})_{n=1}^\infty$ for the zeros $\vec\xi_{n,N}$ of the $n$-th orthogonal polynomial $p_{n,N}$ of the potential $V$ always has a convergent subsequence which tends in the limit $N\to\infty$, $\frac nN\to x\in(0,1]$ to a measure $\sigma_x$ which is supported on a tree like graph whose edges are smooth one-dimensional manifolds, except for the degenerate case $V(z)=\frac1{t_0}|z|^2$.
\end{conjecture}


\subsection{Connection to the Toda lattice hierarchy}
We define the operators $L_N$, $A_N$ and $A^{(k)}_N$ on the completion $\mathcal H$ of the inner product space of all polynomials endowed with the scalar product $(\cdot,\cdot)_N$, given by equation \eqref{e:innerProduct}, through their actions
\[ L_Nq_{n,N}(z)=zq_{n,N}(z),\quad A_Nq_{n,N}(z)=\partial_zq_{n,N}(z),\quad A_N^{(k)}q_{n,N}(z)=\partial_{t_k}q_{n,N}(z) \]
on the orthonormal polynomials $q_{n,N}$, $n\in\N_0$, for a potential $V$ fulfilling the assumptions of Theorem \ref{t:equilibrium measure of a polynomial curve}.

Writing $O_{m,n,N}$ for the matrix elements $(q_{m,N},O_Nq_{n,N})_N$ of some operator $O_N$, we see that
\begin{equation}\label{e:matrixElementsSupp}
L_{m+2,n,N}=0,\quad A_{m,n,N}=0,\quad\textrm{and}\quad A^{(k)}_{m+1,n,N}=0\quad\textrm{for}\quad m\ge n.
\end{equation}

To get the asymptotic behaviour of the orthonormal polynomials, we use Lemma \ref{l:asymptoticOgPol} and will therefore again consider a cut-off $D$ with twice continuously differentiable boundary.

\begin{lemma}
For all $m,n,N\in\N_0$, we have
\[ |L_{m,n,N}|\le\max_{z\in D}|z|. \]
\end{lemma}
\begin{proof}
With the Cauchy-Schwarz inequality, we find
\begin{align*}
|L_{m,n,N}| &= |(q_{m,N},zq_{n,N})_N| \le \sqrt{(q_{m,N},q_{m,N})_N}\,\sqrt{(zq_{n,N}(z),zq_{n,N})_N} \le \max_{z\in D}|z|.\qedhere
\end{align*}
\end{proof}

\begin{lemma}\label{l:borderEstimate}
In the limit $N\to\infty$, there exists an $\varepsilon>0$, depending only on the potential $V$ and the cut-off $D$, such that uniformly in $m,n\le N$
\begin{equation}\label{e:borderEstimate}
\oint_{\partial D}\overline{q_{m,N}(z)}q_{n,N}(z)\e^{-NV(z)}\d z=\ord(\e^{-N\varepsilon}).
\end{equation}
\end{lemma}
\begin{proof}
According to Lemma \ref{l:asymptoticOgPol} and the proof of Theorem \ref{t:equilibrium measure of a polynomial curve}, the norms $(h_{m,N})_{m=0}^\infty$ of the orthogonal polynomials $(p_{m,N})_{m=0}^\infty$ for the potential $V$ fulfil in the limit $N\to\infty$, $\frac mN\to x\in(0,1]$:
\[ h_{m,N}\ge\int_{D\setminus D_{x,+}}|p_{m,N}(z)|^2\e^{-NV(z)}\dd z=\int_{D\setminus D_{x,+}}\e^{-mE_x(z)+\ord(m)}\dd z=\e^{-mE_x(0)+\ord(m)}, \]
where $D_{x,+}$ denotes the support of the equilibrium measure $\mu_x$ for the potential $\frac1xV$ and
\[ E_x(z) = \frac1xV(z)+2\int\log|z-\zeta|\I\d\mu_x(\zeta). \]

For the orthonormal polynomials $q_{n,N}=\frac1{\sqrt{h_{n,N}}}p_{n,N}$, $n\in\N_0$, we therefore get in the limit $N\to\infty$, $\frac mN\to x\in(0,1]$, $\frac nN\to y\in(0,1]$
\begin{align*}
\lim_{N\to\infty}\frac1N\log\left|\oint_{\partial D}\overline{q_m(z)}q_n(z)\e^{-NV(z)}\d z\right|\le-\frac12\min_{z\in\partial D}(x(E_x(z)-E_x(0))+y(E_y(z)-E_y(0))), 
\end{align*}
which by the proof of Theorem \ref{t:equilibrium measure of a polynomial curve} is for all $x,y\in(0,1]$ strictly negative. Moreover, since
\[ \lim_{x\to0}x(E_x(z)-E_x(0))=V(z)+\frac2{\pi t_0}\lim_{x\to0}\int_{D_{x,+}}\log\left|1-\frac z\zeta\right|\I\dd\zeta = V(z)>0 \]
for all $z\in\partial D$, there exists an $\varepsilon>0$ such that relation \eqref{e:borderEstimate} holds.
\end{proof}

\begin{proposition}\label{p:recursionRelation}
We have the operator identity
\begin{equation}\label{e:recursionRelation}
L^*_N = \frac{t_0}NA_N+\sum_{k=1}^{d+1}kt_kL_N^{k-1}+C_N(D),
\end{equation}
where the matrix elements $C_{m,n,N}(D)$ of the operator $C_N(D)$ tend in the limit $N\to\infty$ uniformly in $m,n\le N$ to zero faster than $\e^{-N\varepsilon}$ for some $\varepsilon>0$, depending only on the potential $V$ and the cut-off $D$.
\end{proposition}
\begin{proof}
Writing the multiplication by $z$ as differential operator acting on $\e^{-NV(z)}$ and integrating by parts, we get
\begin{align*}
(q_{m,N},zq_{n,N})_N&=\frac{t_0}N(\partial_zq_{m,N},q_{n,N})_N+\sum_{k=1}^{d+1}k\bar t_k(z^{k-1}q_{m,N},q_{n,N})_N \\
&\qquad+\frac1{2\i}\oint_{\partial D}\overline{q_{m,N}(z)}q_{n,N}(z)\e^{-NV(z)}\d z,
\end{align*}
where the last term is by Lemma \ref{l:borderEstimate} for $m,n\le N$ of the form $\ord(\e^{-N\varepsilon})$ for some $\varepsilon>0$.
\end{proof}

Let us now introduce the projection $\Pi_k$ from the operators on $\mathcal H$ onto the $(k+1)\times(k+1)$ matrices defined by
\[ \Pi_k(O) = (O_{m,n})_{m,n=0}^k\quad\textrm{for any operator}\;O. \]

\begin{corollary}
The operators $L_N$ and $L^*_N$ fulfil in the case $d\ge1$
\begin{equation}\label{e:stringEq}
\Pi_{N-d}([\Pi_N(L_N^*),\Pi_N(L_N)])=\frac{t_0}N+\tilde C_{N-d}(D), 
\end{equation}
where the matrix elements $\tilde C_{m,n,N}(D)$ of the operator $\tilde C_N(D)$ tend in the limit $N\to\infty$ uniformly in $m,n\le N$ faster to zero than $\e^{-N\varepsilon}$ for some $\varepsilon>0$, depending only on the potential $V$ and the cut-off $D$.
\end{corollary}
\begin{proof}
Because of the relations \eqref{e:matrixElementsSupp}, we have for any operator $O$ and any $k\in\N$ the identities
\[ \Pi_{N-k}(\Pi_N(O)\Pi_N(L_N^k)) = \Pi_{N-k}(OL_N^k)\quad\textrm{and}\quad\Pi_N(O)\Pi_N(A) = \Pi_N(OA). \]
Substituting now the relation \eqref{e:recursionRelation} for $L_N^*$, we get
\begin{align*}
\Pi_{N-d}([\Pi_N(L_N^*),\Pi_N(L_N)])&=\frac{t_0}N\Pi_{N-d}([A_N,L_N])+\sum_{k=1}^{d+1}kt_k\Pi_{N-d}([L_N^{k-1},L_N]) \\
&\qquad\;\,+\Pi_{N-d}([\Pi_N(C_N(D)),\Pi_N(L_N)])\\
&= \frac{t_0}N+\Pi_{N-d}([\Pi_N(C_N(D)),\Pi_N(L_N)]).
\end{align*}
With $\tilde C_{N-d}(D)=\Pi_{N-d}([\Pi_N(C_N(D)),\Pi_N(L_N)])$, this gives equation \eqref{e:stringEq}.
\end{proof}

\begin{proposition}
For $1\le k\le d+1$, the operators $L_N$ and $L_N^*$ fulfil the equations
\begin{align}
&\frac{\partial L_N}{\partial t_k}=\frac N{t_0}[L_N,M^{(k)}_N],&&\frac{\partial L_N^*}{\partial t_k}=\frac N{t_0}[L_N^*,M^{(k)}_N],\label{e:todaLax}\\
&\frac{\partial L_N}{\partial\bar t_k}=\frac N{t_0}[M^{(k)*}_N,L_N],&&\frac{\partial L_N^*}{\partial\bar t_k}=\frac N{t_0}[M^{(k)*}_N,L_N^*],\label{e:todaLax2}
\end{align}
where, with the notation $O_+$ for the lower triangular, $O_-$ for the upper triangular, and $O_0$ for the diagonal part of the operator $O$,
\begin{equation}\label{e:todaHamiltonianOperator}
M^{(k)}_N=(L^k_N)_++\frac12(L^k_N)_0,\quad M^{(k)*}_N=((L^*_N)^k)_-+\frac12((L^*_N)^k)_0.
\end{equation}
\end{proposition}
\begin{proof}
For all $n\in\N_0$, we have
\[ \sum_{\ell,m=0}^\infty A^{(k)}_{m,n,N}L_{\ell,m,N}q_{\ell,N}=\frac{\partial}{\partial t_k}(L_Nq_{n,N})=\sum_{\ell=0}^\infty\frac{\partial L_{\ell,n,N}}{\partial t_k}q_{\ell,N}+\sum_{\ell,m=0}^\infty L_{m,n,N}A^{(k)}_{\ell,m,N}q_{\ell,N}, \]
and so $\frac{\partial L_N}{\partial t_k}=[L_N,A^{(k)}_N]$.
Now, because of
\[ A^{(k)}_{m,n,N}=\left(q_{m,N},\frac{\partial q_{n,N}}{\partial t_k}\right)_N=-\left(\frac{\partial q_{m,N}}{\partial\bar t_k},q_{n,N}\right)_N-\frac N{t_0}(q_{m,N},z^kq_{n,N})_N \]
for all $m,n\in\N_0$, we have $(A^{(k)}_N)_-=-\frac N{t_0}(L_N^k)_-$, and introducing the orthogonal polynomials $p_{n,N}=\sqrt{h_{n,N}}q_{n,N}$, $n\in\N_0$, for $V$, we find
\[ \left(q_{n,N},\frac{\partial q_{n,N}}{\partial t_k}\right)_N = \sqrt{h_{n,N}}\frac{\partial}{\partial t_k}\frac1{\sqrt{h_{n,N}}} = \left(\frac{\partial q_{n,N}}{\partial\bar t_k},q_{n,N}\right)_N\quad\textrm{for all $n\in\N_0$,} \]
which implies $(A^{(k)}_N)_0=-\frac N{2t_0}(L_N^k)_0$. So, by setting $M^{(k)}_N=L_N^k+\frac{t_0}NA^{(k)}_N$, we get the relations \eqref{e:todaHamiltonianOperator} and fulfil the first equation of \eqref{e:todaLax}. Remarking that
\[ M^{(k)*}_{m,n,N} = -\frac{t_0}N\left(q_{m,N},\frac{\partial q_{n,N}}{\partial\bar t_k}\right)_N\quad\textrm{for all $m,n\in\N_0$,} \]
we see that also the first equation of \eqref{e:todaLax2} and therefore, by taking the adjoint equations, all four equations of \eqref{e:todaLax} and \eqref{e:todaLax2} are fulfilled.
%
\end{proof}

Again, as in the case of the polynomial curves, only the equations for the times $(t_k)_{k=1}^{d+1}$ need to be considered. Moreover, equation \eqref{e:stringEq} provides us with a dispersionful analogon of the string equation.

\subsection{Application}\label{s:application}
Let us for some $d\in\N$ consider the potential
\[ V(z)=\frac1{t_0}(|z|^2-t_{d+1}(z^{d+1}+\bar z^{d+1})),\quad t_{d+1}>0, \]
on a cut-off $D$ chosen as disc centered in the origin such that $V(z)>0$ for all $z\in D\setminus\{0\}$. Additionally, for $d=1$, we require $t_2<\frac12$.

Using the weighted homogeneity of the universal polynomials $P_{j,k}$ in equation \eqref{e:harmonicMoments}, we find for small $t_0>0$ the parametrisation $h$ of a polynomial curve with only non-vanishing harmonic moment $t_{d+1}$ and encircled area $\pi t_0$:
\begin{equation}\label{e:correspondingPolCurve}
h(w)=rw+aw^{-d},\quad a=(d+1)t_{d+1}r^d,
\end{equation}
where $r$ is the smallest positive solution (which does exist for small $t_0$) of the algebraic equation
\begin{equation}\label{e:correspondingPolCurve1}
t_0=r^2-d(d+1)^2t_{d+1}^2r^{2d}.
\end{equation}
\begin{proposition}\label{p:recursionOnPol}
There exists an $\varepsilon>0$ such that the orthonormal polynomials $(q_{n,N})_{n=0}^N$ for the potential $V$ fulfil (with $q_{-n,N}=0$ for $n\in\N$) for $0\le n\le N-1$ the recursion formula
\begin{equation}\label{e:recursionOnPol}
zq_{n,N}=r_{n+1,N}q_{n+1,N}+a_{n-d,N}q_{n-d,N}+\sum_{j\ge2}\ord(\e^{-N\varepsilon})q_{n+1-j(d+1),N},
\end{equation}
where
\begin{equation}\label{e:recursionCoefficient}
\bar a_{n,N}=(d+1)t_{d+1}\prod_{k=1}^dr_{n+k,N}+\ord(\e^{-N\varepsilon}),\quad 0\le n\le N-d,
\end{equation}
and (setting $r_{-n,N}=0$ for $n\in\N_0$) the $r_{n,N}$ are recursively defined by 
\begin{equation}\label{e:recursionRelationCoefficients}
|a_{n,N}|^2=r_{n+1,N}^2-\sum_{k=1}^{d-1}|a_{n-d+k}|^2-\frac{t_0}N(n+1)+\ord(\e^{-\frac12N\varepsilon}),\quad 0\le n\le N-d.
\end{equation}
\end{proposition}
\begin{proof}
Proposition \ref{p:recursionRelation} together with the restriction \eqref{e:matrixElementsSupp} shows us that for $m,n\le N$ only the matrix elements $L_{m,n,N}$ with $-1\le n-m\le d$ are not zero or at least exponentially decreasing for $N\to\infty$. Moreover, the symmetry $z\mapsto\e^{\frac{2\pi\i}{d+1}}z$ of the potential $V$ implies that the orthonormal polynomials have the form 
\[ q_{j(d+1)+r,N}(z)=z^r\tilde q_{j(d+1)+r,N}(z^{d+1}),\quad 0\le r\le d,\quad j\in\N_0, \]
for some polynomials $\tilde q_{j(d+1)+r,N}$ each with degree $j$. Therefore, the only non-vanishing matrix elements are $L_{n+1-j(d+1),n,N}$, $j\in\N_0$.

Setting now $r_{n+1,N}=L_{n+1,n,N}$ and $a_{n,N}=L_{n,n+d,N}$, the matrix element $(n+d,n)$ of the equality \eqref{e:recursionRelation} gives us relation \eqref{e:recursionCoefficient}.

The diagonal part of equation \eqref{e:stringEq} now reads
\[ r_{m+1,N}^2-r_{m,N}^2=\frac{t_0}N+|a_{m,N}|^2-|a_{m-d,N}|^2+\ord(\e^{-N\varepsilon}),\quad 0\le m\le N-d. \]
Adding up the first $n+1$ of these equations gives us \eqref{e:recursionRelationCoefficients}.
\end{proof}
In the special case $d=1$, we may set $D=\C$ so that the error terms disappear. Then $a_{n,N}=2t_2r_{n+1,N}$ and the recursion relation \eqref{e:recursionRelationCoefficients} simplifies to
\[ (1-4t_2^2)(r_{n+1,N}^2-r_{n,N}^2)=\frac{t_0}N,\quad n\in\N_0, \]
which has the explicit solution $r_{n,N}=\sqrt{\frac{t_0n}{N(1-4t_2^2)}}$, $n\in\N$.

\begin{lemma}\label{l:zeros}
Let $(q_n)_{n\in\N_0}$ be real polynomials with $\deg q_n=n$, $q_0>0$ and
\begin{equation}\label{e:generalRecursionRelation}
zq_n(z)=r_{n+1}q_{n+1}(z)+a_{n-d}q_{n-d}(z),
\end{equation}
for some $d\in\N$ and some positive numbers $r_n$ and $a_n$, $n\in\N_0$. Then there uniquely exist real polynomials $\tilde q_{j(d+1)+r}$ of degree $j$ with 
\[ q_{j(d+1)+r}(z)=z^r\tilde q_{j(d+1)+r}(z^{d+1}),\quad 0\le r\le d,\quad j\in\N_0, \]
and these $\tilde q_{j(d+1)+r}$ have exactly $j$ different, real, positive zeros.
\end{lemma}
\begin{proof}
Rewriting the recursion relation \eqref{e:generalRecursionRelation} in terms of the polynomials $\tilde q_n$, we get
\begin{align*}
&r_{n+1}\tilde q_{n+1}(z)=\phantom{z}\tilde q_n(z)-a_{n-d}\tilde q_{n-d}(z)\quad\textrm{for}\quad n\notin (d+1)\N_0+d\quad\text{and}\\
&r_{n+1}\tilde q_{n+1}(z)=z\tilde q_n(z)-a_{n-d}\tilde q_{n-d}(z)\quad\textrm{for}\quad n\in (d+1)\N_0+d.
\end{align*}
Assume now by induction in $j\in\N_0$ that all the zeros $(\xi_{k,n})_{k=1}^{\deg\tilde q_n}$ of $\tilde q_n$ are positive numbers for $j(d+1)\le n<(j+1)(d+1)$ and fulfil
\[ \xi_{k,n}<\xi_{k,n+1}<\xi_{k+1,j(d+1)}\quad\text{for}\quad j(d+1)\le n<j(d+1)+d\quad\textrm{and}\quad1\le k\le j. \]
(To simplify the notation, we set $\xi_{0,n}=0$ and $\xi_{\deg\tilde q_n+1,n}=\infty$ for all $n\in\N_0$.)

Then, for $0\le k\le j$,
\[ \sign(\tilde q_{(j+1)(d+1)}(\xi_{k,j(d+1)+d}))=-\sign(\tilde q_{(j+1)(d+1)}(\xi_{k+1,j(d+1)})), \]
and so we have $\xi_{k,n}<\xi_{k+1,(j+1)(d+1)}<\xi_{k+1,n}$ for all $0\le k\le j$ and $j(d+1)\le n<(j+1)(d+1)$.

Now we see recursively in $1\le r\le d$ that
\[ \sign(\tilde q_{(j+1)(d+1)+r}(\xi_{k,(j+1)(d+1)+r-1}))=-\sign(\tilde q_{(j+1)(d+1)+r}(\xi_{k,j(d+1)+r})),\quad 0\le k\le j, \]
verifying the induction hypothesis.
\end{proof}
Provided now that the error terms in recursion relation \eqref{e:recursionOnPol} can be neglected (which is what we expect), we can apply Lemma \ref{l:zeros} and see that in the limit $N\to\infty$ the zeros of the orthonormal polynomials will lie on the tree-like graph $\{z\in\C\;|\;z^{d+1}\in[0,\infty)\}$. As shown in Section \ref{s:zerosOgPol}, their density in the continuum limit can therefore be determined as the discontinuity of the Schwarz function of the polynomial curve parametrised by $h$. In the cases $d=1$, where $h$ parametrises an ellipse, and $d=2$, where $h$ parametrises a hypotrochoid, this can explicitly be done.
\begin{proposition}\label{p:schwarzFunctionEllipse}
The Schwarz function of the ellipse $\{(x,y)\in\R^2\;|\;\frac{x^2}{b_1^2}+\frac{y^2}{b_2^2}=1\}$ with half-axis $b_1,b_2>0$ is given by
\begin{equation}\label{e:schwarzFunctionEllipse}
S(z) = \frac{b_1^2+b_2^2}{b_1^2-b_2^2}\,z-\frac{2b_1b_2z}{b_1^2-b_2^2}\sqrt{1-\frac{b_1^2-b_2^2}{z^2}}.
\end{equation}
\end{proposition}
\begin{proof}
Solving the quadratic equation
\[ \frac{(z+S(z))^2}{b_1^2}-\frac{(z-S(z))^2}{b_2^2}=4 \]
for $S(z)$ and choosing the sign such that $S(b_1)=b_1$, leads us immediately to equation \eqref{e:schwarzFunctionEllipse}.
\end{proof}

\begin{proposition}
Let $\vec\xi_{n,N}$ be the zeros of the orthonormal polynomial $q_{n,N}$ for the potential 
\[ V(z)=\frac1{t_0}(|z|^2-2t_2(z^2+\bar z^2)),\quad 0<t_2<\frac12. \]
Then, for $t_0>0$ small enough, each converging subsequence of the sequence $(\delta_{\vec\xi_{n,N}})$ tends for $N\to\infty$ and $\frac nN\to x$ to the measure $\sigma$ on $\R$ with $\d\sigma(y)=\rho(y)\d y$ and 
\begin{equation}\label{e:densityEllipse}
\rho(y)=\frac2{(b_1^2-b_2^2)\pi}\sqrt{b_1^2-b_2^2-y^2}\;\chi_{[-\sqrt{b_1^2-b_2^2},\sqrt{b_1^2-b_2^2}]}(y).
\end{equation}
Here $b_1$ and $b_2$ denote the half-axis of the ellipse with vanishing first harmonic moment, second harmonic moment $t_2$, and area $\pi xt_0$. 
\end{proposition}
\begin{proof}
Proposition \ref{p:recursionOnPol} together with Lemma \ref{l:zeros} tells us that the zeros all lie on the real axis. The density $\rho$ is therefore given by 
\[ \rho = \frac{\delta S}{2\pi\i xt_0}, \]
where $\delta S$ is the discontinuity of the Schwarz function $S$ of the ellipse, which was calculated in Proposition \ref{p:schwarzFunctionEllipse}. Here the branch cut of $S$ has to be chosen on the part of the real axis where $\delta S$ is purely imaginary, that is on the interval $[-\sqrt{b_1^2-b_2^2},\sqrt{b_1^2-b_2^2}]$ between the two foci of the ellipse. This directly brings us to equation \eqref{e:densityEllipse}.
\end{proof}

\begin{proposition}\label{p:cubicS}\mbox{}
\begin{enumerate}
\item
The function $h(w)=rw+aw^{-2}$, $w\in S^1$, is for all $r>0$ and $a\in\R$ with $2|a|<r$ a parametrisation of a polynomial curve, the hypotrochoid.
\item
The Schwarz function $S$ of this hypotrochoid is given by
\[ S(3rz) = \left((a^2-r^2)\,\omega\left({\textstyle\frac{4rz^3}a}\right)^2+(2a^2+r^2)\,\omega\left({\textstyle\frac{4rz^3}a}\right)+a^2+2r^2\right)\frac{z^2}a, \]
where
\[ \omega(y)=\sqrt[3]{1-\frac2y\left(1+\sqrt{1-y}\right)}+\sqrt[3]{1-\frac2y\left(1-\sqrt{1-y}\right)} \]
and the branch cuts of the roots are chosen on the negative real axis.
\end{enumerate}
\end{proposition}
\begin{proof}
Because the derivative of $h$ vanishes only for those $w$ with $rw^3=2a$, which for $2|a|<r$ all lie inside the unit disc, $h$ is a biholomorphic map from the exterior of the unit disc and is therefore a parametrisation of a polynomial curve.

Solving now $h(w)=z$ for $w$ leads us to the cubic equation
\[ \tilde w^3-3p\tilde w-2q=0,\quad p=\left(\frac z{3r}\right)^2,\quad q=\left(\frac z{3r}\right)^3-\frac a{2r}, \]
for $\tilde w=w-\frac z{3r}$, whose solution is of the form
\[ \tilde w=\sqrt[3]{q+\sqrt{q^2-p^3}}+\sqrt[3]{q-\sqrt{q^2-p^3}}. \]
Choosing the branch cuts of the cubic root such that the solution is continuous outside the unit disc, gives us the inverse function for $h$:
\[ h\I(z)=\left(1+\omega\left({\textstyle\frac{4r}a\left(\frac z{3r}\right)^3}\right)\right)\frac z{3r}. \]
So the Schwarz function $S$ is given by
\begin{align*}
S(z)&=\frac r{h\I(z)}+ah\I(z)^2=\frac{a^2-r^2}ah\I(z)^2+\frac{rz}ah\I(z) \\
&=\left[(a^2-r^2)\,\omega\left({\textstyle\frac{4r}a\left(\frac z{3r}\right)^3}\right)^2+(2a^2+r^2)\,\omega\left({\textstyle\frac{4r}a\left(\frac z{3r}\right)^3}\right)+a^2+2r^2\right]\frac1a\left(\frac z{3r}\right)^2.\qedhere
\end{align*}
\end{proof}

\begin{figure}
\includegraphics[width=0.75\textwidth]{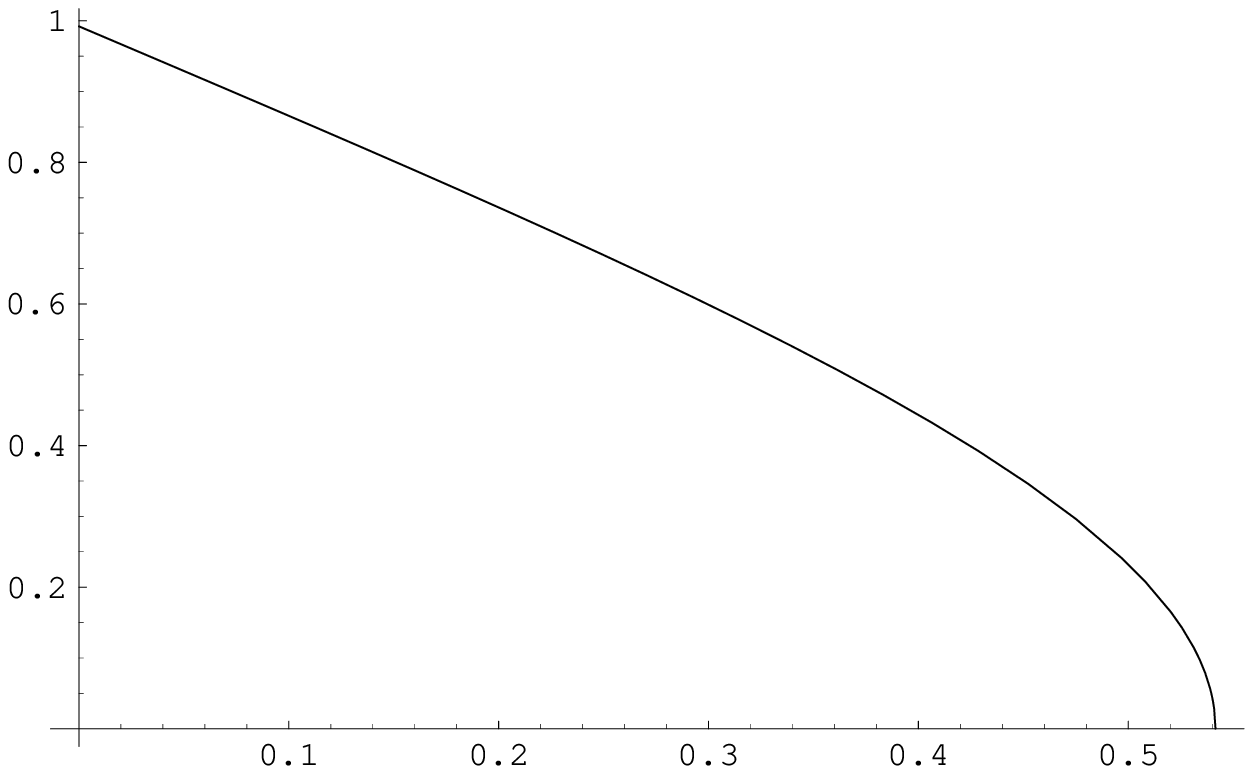}
\caption{Graph of $\sigma(z) = \rho(\frac{4r}a(\frac z{3r})^3)\frac1a(\frac z{3r})^2$ for $z\in[0,(\frac{27}4ar^2)^{\frac13}]$ in the case $t_0=\frac14$, $t_3=\frac19$}
\end{figure}

\begin{proposition}
Let $\vec\xi_{n,N}$ be the zeros of the orthonormal polynomial $q_{n,N}$ for the potential 
\[ V(z)=\frac1{t_0}(|z|^2-3t_3(z^3+\bar z^3)),\quad t_3>0. \]
Then, provided the error terms in \eqref{e:recursionOnPol} are neglectable, for $t_0>0$ small enough, each converging subsequence of the sequence $(\delta_{\vec\xi_{n,N}})$ tends for $N\to\infty$ and $\frac nN\to x$ to the measure $\sigma$,
\begin{equation}\label{e:cubicDensityZeros}
\d\sigma = \rho\left(\textstyle{\frac{4r}a\left(\frac z{3r}\right)^3}\right)\frac1a\left(\frac z{3r}\right)^2\d z,
\end{equation}
where 
\begin{equation}\label{e:coeffPolCurve}
r=\frac1{6t_3}\sqrt{1-\sqrt{1-72xt_0t_3^2}},\quad a=\frac1{12t_3}\left(1-\sqrt{1-72xt_0t_3^2}\right),
\end{equation}
$\d z$ denotes the infinitesimal line element on the one-dimensional manifold $\{z\in\C\;|\;z^3\in(0,\frac{27}4ar^2)\}$, 
\[ \rho(y)=\left((\omega_+(y)^2-\omega_-(y)^2)(r^2-a^2)-(\omega_+(y)-\omega_-(y))(r^2+2a^2)\right)\frac{\sqrt3}{2\pi(r^2-2a^2)}, \]
and
\[ \omega_\pm(y)=\sqrt[3]{\frac2y\left(1\pm\sqrt{1-y}\right)-1}. \]
\end{proposition}
\begin{proof}
For small $t_0$, the support of the equilibrium measure for the potential $V$ is given by the polynomial curve $\gamma$ defined by the parametrisation \eqref{e:correspondingPolCurve} with $d=2$. Solving now the biquadratic equation \eqref{e:correspondingPolCurve1} for $r$, we find, substituting $xt_0$ for $t_0$, the relations \eqref{e:coeffPolCurve}.

Proposition \ref{p:cubicS} provides us now with the Schwarz function $S$ for the curve $\gamma$. The function $S$ is discontinuous in those $z\in\C$, where 
\[ 1-\frac2y(1\pm\sqrt{1-y})<0,\quad y=\frac{4r}a\left(\frac z{3r}\right)^3, \]
that is for $y\in(0,1)$ resp. $z^3\in(0,\frac{27}4ar^2)$.

Because we know that all zeros of the orthonormal polynomials $q_{n,N}$ lie on the set $\{z\in\C\;|\;z^3\in[0,\infty)\}$, the limiting measure $\sigma$ is given through equation \eqref{e:densityZeros} leading us to formula \eqref{e:cubicDensityZeros}.
\end{proof}





\end{document}